\documentclass[10pt]{amsart}

\usepackage[top=1in,bottom=1in,left=1in,right=1in,heightrounded]{geometry}
\usepackage{setspace}
\usepackage[parfill]{parskip}
\allowdisplaybreaks

\usepackage{amsmath,amsthm,amssymb} 
\usepackage{mathtools}
\usepackage{mathrsfs}
\usepackage{stmaryrd}

\usepackage{tikz}
\usepackage{tikz-cd}
\usetikzlibrary{decorations.pathmorphing,decorations.markings,patterns,hobby,arrows.meta}
\tikzset{>=stealth}
\usepackage{pgfplots}
    \pgfplotsset{compat=1.18}
\usepackage{graphicx}
\usepackage{diagbox}
\usepackage{adjustbox}

\usepackage{array}
\usepackage{makecell}

\usepackage{longtable}
\usepackage{nicematrix}

\usepackage{enumitem}

\usepackage{caption} 
\usepackage{multicol}
\usepackage{varwidth}

\usepackage{tcolorbox}

\usepackage[dvipsnames]{xcolor}
\usepackage{colortbl}
\usepackage{soul}
\usepackage{todonotes}

\usepackage{hyperref}
\usepackage[capitalise]{cleveref}
\usepackage{url}
\hypersetup{
   colorlinks,
   citecolor=blue,
   filecolor=blue,
   linkcolor=blue,
   urlcolor=blue
}

\usepackage[numbers]{natbib}
\usepackage{lmodern}
\usepackage{mathpazo}
\usepackage[cal=boondox,scr=boondox]{mathalfa}

\AtBeginDocument{
  \DeclareSymbolFont{AMSb}{U}{msb}{m}{n}
  \DeclareSymbolFontAlphabet{\mathbb}{AMSb}}

\DeclareMathAlphabet{\mathbx}{U}{BOONDOX-ds}{m}{n}
\SetMathAlphabet{\mathbx}{bold}{U}{BOONDOX-ds}{b}{n}
\DeclareMathAlphabet{\mathbbx}{U}{BOONDOX-ds}{b}{n}

\SetMathAlphabet{\mathcal}{bold}{U}{dutchcal}{b}{n}
\DeclareMathAlphabet{\mathbcal}{U}{dutchcal}{b}{n}

\usepackage{fancyhdr}

\theoremstyle{plain}
\newtheorem{theorem}{Theorem}[section]
\newtheorem{lemma}[theorem]{Lemma}
\newtheorem{corollary}[theorem]{Corollary}
\newtheorem{proposition}[theorem]{Proposition}
\newtheorem{conjecture}[theorem]{Conjecture}

\theoremstyle{definition}

\newtheorem{remark}[theorem]{Remark}
\newtheorem{definition}[theorem]{Definition}

\newcommand{\zz}{\mathbx Z}   
\newcommand{\qq}{\mathbx Q}   
\newcommand{\ff}{\mathbx F}   
\newcommand{\rr}{\mathbx R}   
\newcommand{\cc}{\mathbx C}   
\newcommand{\kk}{\mathbx k}   

\newcommand{\oo}{\mathscr O}  

\newcommand{\Span}{\operatorname{span}}      

\newcommand{\set}[1]{\left\{#1\right\}}     
\newcommand{\setp}[2]{\left\{#1\ \middle|\ #2\right\}} 

\newcommand{\Heleq}{\operatorname{\preccurlyeq}}       

\newcommand\marker[1]{\rlap{\raisebox{-0.5ex}{\tiny(#1)}}\def\@currentlabel{(#1)}}

\numberwithin{equation}{subsection}

\renewcommand{\epsilon}{\varepsilon}
\renewcommand{\phi}{\varphi}
\renewcommand{\emptyset}{\varnothing}
\renewcommand{\geq}{\geqslant}
\renewcommand{\leq}{\leqslant}

\renewcommand{\gcd}{\operatorname{GCD}}

\colorlet{deewangcolor}{cyan!50}

\colorlet{emeraldcolor}{blue!30}

\colorlet{heidicolor}{magenta!50}

\colorlet{miacolor}{yellow!90}

\colorlet{sandracolor}{green!40!yellow}

\colorlet{stevencolor}{orange!60}

\title{Ekedahl-Oort strata in the $\boldsymbol{\mathsf{GU}(q-2,2)}$ Shimura variety}
\subjclass[2020]{11G18, 14G35, 11G10}

\author{Emerald Andrews} 
\address{Department of Mathematics and Computer Science, Washington College, Chestertown, Maryland 21620, USA}
\email{estacy2@washcoll.edu}

\author{Deewang Bhamidipati} 
\address{Department of Mathematics and Statistics, Carleton College, Northfield, MN 55057, USA} 
\email{bdeewang@carleton.edu}

\author{Maria Fox} 
\address{Department of Mathematics, Oklahoma State University, Stillwater, OK 74078, USA}
\email{maria.fox@okstate.edu}

\author{Heidi Goodson} 
\address{Department of Mathematics, Brooklyn College and The Graduate Center, City University of New York, Brooklyn, NY 11210 USA}
\email{heidi.goodson@brooklyn.cuny.edu}

\author{Steven R. Groen} 
\address{Korteweg-de Vries Institute for Mathematics, University of Amsterdam, Amsterdam, Netherlands}
\email{s.r.groen@uva.nl}

\author{Sandra Nair} 
\address{Department of Mathematics, Colorado State University, Fort Collins, CO 80523, USA}
\email{sandra.nair@colostate.edu}

\begin{document}

\begin{abstract}
    This paper concerns the characteristic-$p$ fibers of $\mathsf{GU}(q-2,2)$ Shimura varieties, which classify abelian varieties with additional structure. These Shimura varieties admit two stratifications of interest: the Ekedahl-Oort stratification, based on the isomorphism class of the $p$-torsion subgroup scheme, and the Newton stratification, based on the isogeny class of the $p$-divisible group.  In this paper, we present several novel techniques that give a better understanding of the Ekedahl-Oort stratification and of the interaction between the two stratifications for a general signature $(q-2,2)$.
\end{abstract}

\maketitle

\section{Introduction}\label{sec:intro}

This paper studies the characteristic-$p$ fibers of Shimura varieties of PEL type. Shimura varieties of PEL type are moduli spaces of abelian varieties with additional structure, and this moduli interpretation gives rise to two stratifications of the characteristic-$p$ fiber: the Newton stratification and the Ekedahl-Oort stratification. The Ekedahl-Oort stratification is based on the isomorphism class of the $p$-torsion subgroup scheme of the parameterized abelian varieties, while the Newton stratification is based on the isogeny class of their $p$-divisible group. Though each stratification is compelling in its own right, it is also interesting to study the interactions between the two stratifications. For instance, the two stratifications coincide on the modular curve, both distinguishing between \emph{ordinary} and \emph{supersingular} elliptic curves. 

When the modular curve is replaced by a general Shimura variety of PEL type, the interaction between the Ekedahl-Oort and Newton stratifications is very subtle; for example, see the results in \cite{ViehmannWedhorn2013}.

More precise statements can be made when one restricts attention to a smaller class of Shimura varieties. We study the interaction between these two stratifications for \emph{unitary Shimura varieties} $\mathcal{M}(q-b, b)$. These are moduli spaces of abelian varieties of dimension $q$ with an action of an order in an imaginary quadratic field $K$ that meets the ``signature $(q-b, b)$" condition.  In addition to the main focus of her thesis, especially under the assumption of the relevant prime $p$ being split in $K$, Wooding \cite{wooding_2016} serves as an excellent source of background information on the Newton and Ekedahl-Oort stratifications of $\mathcal{M}(q-b, b)$. 

The unitary Shimura varieties of signature $(q,0)$ are zero-dimensional and those of signature $(q-1,1)$ have been extensively studied \cite{bultel2006congruence,VollaardWedhorn}, we therefore focus on unitary Shimura varieties of signature $(q-2,2)$. Unlike those of smaller signatures, the Shimura varieties $\mathcal{M}(q-2,2)$ are in general not of Coxeter type (see \cite{goertzhe} for details). For this reason, the interaction between Ekedahl-Oort and Newton strata is much more complicated than in previously studied cases, and is also likely a better reflection of the situation for general signature $(q-b,b)$. 

In this paper we study three aspects of the Ekedahl-Oort stratification of $\mathcal{M}(q-2,2)$, under the assumption that the relevant prime $p$ is inert in $K$, in order to develop tools for studying the interaction between the Ekedahl-Oort stratification and Newton stratification.

First, in \cref{sec:eo}, we study the topological closure relations among the Ekedahl-Oort strata. In \cref{thm:primaryHerelations,thm:secondaryHerelations}, we prove that a number of topological closure relations hold. While topological closure relations between Ekedahl-Oort strata of Shimura varieties of PEL type have been previously studied \cite{PinkWedhornZiegler2011,HeOrder2007}, we are able to express closure relations for $\mathcal{M}(q-2,2)$ much more explicitly than is possible in the general setting.

Second, in \cref{sec:prod}, we relate $\mathcal{M}(q-2,2)$  to simpler unitary Shimura varieties via the product maps 
$$\mathcal{M}(a,b) \times \mathcal{M}{(c,d}) \to \mathcal{M}(q-2,2),$$
induced by taking products of the parameterized abelian varieties. In \cref{thm:1x1-complete,thm2x0}, we explicitly describe the effect of these product maps on the Ekedahl-Oort stratifications. These results give a solution to the analogue of a long-standing question of Moonen and Oort \cite[Question~11]{openproblems}, when the Siegel modular variety is replaced by $\mathcal{M}(q-2,2)$.

Third, in \cref{sec:smv}, we relate $\mathcal{M}(q-2,2)$ to the Siegel modular variety via a forgetful map
$$\mathcal{M}(q-2,2) \to \mathcal{A}_q,$$ which ``forgets" the action of the field $K$ on the abelian varieties parameterized by $\mathcal{M}(q-2,2)$.

More precisely, in \cref{thm:SiegelWeyl}, we concretely describe the effect of this map on the Ekedahl-Oort stratifications. As much more is known about the Ekedahl-Oort and Newton stratifications of $\mathcal{A}_q$ than other Shimura varieties, this is an especially powerful tool. 

These three aspects of the Ekedahl-Oort stratification can each be leveraged to illuminate the relationship between the Ekedahl-Oort and Newton stratifications of $\mathcal{M}(q-2,2)$. As a running example, we apply these techniques to make some observations about the interaction between the Ekedahl-Oort stratification and the supersingular locus (the unique closed Newton stratum) of $\mathcal{M}(q-2,2)$; with more care, our techniques are equally well-suited to the study of any Newton stratum. In the tradition of \cite{VollaardWedhorn}, we hope a better understanding of the interaction between the Ekedahl-Oort and Newton strata may be useful when studying more subtle aspects of the geometry of $\mathcal{M}(q-2,2)$.  

We apply the techniques developed in this article to the case of $\mathcal{M}(3,2)$ in \cite{RNTspinoff}.

\subsection*{Acknowledgements}

This project started at the Rethinking Number Theory workshop in June 2022. We want to thank the organizers for the workshop and are grateful for the supportive, collaborative research environment this workshop provided. The workshop was supported by the Number Theory Foundation, the American Institute of Mathematics, and the University of Wisconsin-Eau Claire. Additionally, we would like to thank Rachel Pries, Damiano Testa, and Martin Weissman for useful conversations. We would also like to thank Satyam Patel for programming assistance and  Mohamed Tawfik for his early contributions to the project.

M.F. was supported in part by NSF MSPRF Grant 2103150. H.G. was supported by NSF grant DMS-2201085 and a PSC-CUNY Award, jointly funded by the Professional Staff Congress and CUNY.


\section{Background}\label{sec:background}
\subsection{Unitary Shimura varieties}\label{sec:UnitaryShimuraVariety}

Unitary Shimura varieties are moduli spaces of abelian varieties equipped with extra structure, including an action of an order in an imaginary quadratic field. To define an integral model of such a Shimura variety, we fix a prime $p>2$, a positive integer $q$, non-negative integers $a$ and $b$ such that $a+b=q$, and an imaginary quadratic field $K$. We further assume that the prime $p$ is inert in $K$, and so we identify $\mathcal{O}_K/(p)$ as $\ff_{p^2}$ throughout, where $\mathcal{O}_K$ is the ring of integers of $K$. 

\begin{definition}\label{def:PELdatum}
We use the \textbf{PEL datum} $(K, \oo_K \otimes_{\mathbx{Z} } \mathbx{Z}_{(p)}, {}^*, V, (\cdot,\cdot),\Lambda, \boldsymbol{\mathsf{G}}, h)$ of Kottwitz \cite{Ko}, defined as: 

\begin{itemize}
\item{$K$ is the imaginary quadratic field introduced above,  with ${}^*$ being the nontrivial automorphism of $K$ over $\mathbx{Q}$;}
\item{$V$ is a $K$-vector space of dimension $q$, equipped with  a perfect alternating $\qq$-bilinear pairing $(\cdot,\cdot): V \times V \to \qq$ such that $(xv,w) = (v, x^*w)$ for all $x\in K$ and $v,w \in V$; }
\item{$\boldsymbol{\mathsf{G}}$ is the algebraic group of $K$-linear symplectic similitudes of $(V, (\cdot,\cdot) )$. We assume that $\boldsymbol{\mathsf{G}}_{\mathbx{R}}$ is isomorphic to the real algebraic group $\mathsf{GU}(a,b)$;}
\item{$\Lambda$ is an $\oo_K \otimes_{\mathbx{Z} } \mathbx{Z}_{(p)}$-invariant lattice of $V \otimes_{\mathbx{Q}} \mathbx{Q}_p$ such that the alternating form induced by $(\cdot,\cdot)$ is a perfect $\mathbx{Z}_p$-form;}
\item{$h: \mathrm{Res}_{\cc/\rr}(\mathbx{G}_{m,\cc}) \to \boldsymbol{\mathsf{G}}_\rr$ is the homomorphism of real algebraic groups mapping $z \mapsto \mathrm{diag}(z^a, \overline{z}^b)$. }
\end{itemize}
\end{definition}

Let $L$ be the reflex field associated to the PEL-datum $(K, \oo_K \otimes_{\mathbx{Z} } \mathbx{Z}_{(p)}, *, V, (\cdot,\cdot),\Lambda, \boldsymbol{\mathsf{G}}, h)$; if $a=b$ then $L = \mathbx{Q}$, and $L=K$ otherwise. Let $\mathbx{A}_f^p$ denote the ring of finite adeles with a trivial component at $p$. Fix a compact open subgroup $C^p \subset \boldsymbol{\mathsf{G}}(\mathbx{A}_f^p)$. For $C^p$ small enough, the construction of Kottwitz \cite{Ko} attaches to this PEL datum a smooth, quasi-projective scheme $\mathbf{M}(a,b)_{C^p}$ over $\mathrm{Spec}(\oo_{L,(p)})$ with the following moduli interpretation.

Let $S$ be an $\oo_{L,(p)}$-scheme. Then the set  $\mathbf{M}(a,b)_{C^p}(S)$ parameterizes isomorphism classes of tuples $(A, \iota, \lambda, \xi)$, where:
 
\begin{itemize}
    \item $A$ is an abelian variety over $S$ of dimension $q$;
    \item $\iota: \oo_K \otimes_{\zz} \zz_{(p)}\to \mathrm{End}(A)\otimes_\zz \zz_{(p)}$ is a nonzero homomorphism of $\zz_{(p)}$-algebras such that the Rosati involution on $\mathrm{End}(A)\otimes_\zz \zz_{(p)}$ induces the involution ${}^*$ on $ \oo_K \otimes_{\zz} \zz_{(p)}$;
    \item $\lambda$ is a one-dimensional $\qq$-subspace of $\mathrm{Hom}(A, A^\vee)\otimes_\zz \qq$ that contains a $p$-principal $\oo_K \otimes_{\zz} \zz_{(p)}$-linear polarization;
    \item $\xi: \mathrm{H}_1(A, \mathbx{A}_f^p) \to V \otimes_\qq \mathbx{A}_f^p \ \mathrm{mod } \ C^p$ is a $C^p$-level structure.
\end{itemize}

We also require that $(A,\iota)$ meets Kottwitz's determinant condition of signature $(a,b)$. Two tuples $(A, \iota, \lambda, \xi)$ and $(A', \iota', \lambda', \xi')$ are isomorphic if there exists a prime-to-$p$ isogeny from $A$ to $A'$, commuting with the action of $\oo_K \otimes_{\zz} \zz_{(p)}$, mapping $\xi$ to $\xi'$ and $\lambda$ to $\lambda'$. 

The integral model $\mathbf{M}(a,b)_{C^p}$ has relative dimension of $ab$. The main object of study for this paper is the characteristic $p$  \textbf{unitary Shimura variety}, denoted by $\mathcal{M}(a,b)$, which is the fiber at $p$ of $\mathbf{M}(a,b)_{C^p}$. In particular, $\mathcal{M}(a,b)$ is defined over the residue field $\ff$ of $L$ at $p$ and is of dimension $ab$. Since $\mathbf{M}(a,b)_{C^p} \cong \mathbf{M}(b,a)_{C^p}$, we assume without loss of generality that $0 \leq b \leq a$. 

Frequently, important properties of  $\mathcal{M}(a,b)$ can be understood in terms of its geometric points. For this reason, we fix an algebraic closure $\kk$ of $\ff$.

\subsection{Ekedahl-Oort and Newton Stratifications}

In this section, we recall the definitions of the Ekedahl-Oort and Newton stratifications of $\mathcal{M}(a,b)$. For more details, see \cite{ViehmannWedhorn2013}.

The Ekedahl-Oort stratification is based on the isomorphism class of the $p$-torsion group scheme of the parameterized abelian varieties. Two field-valued points $(A, \iota, \lambda, \xi)$ and $(A', \iota', \lambda', \xi')$ of $\mathcal{M}(a,b)$ are in the same \textbf{Ekedahl-Oort stratum} if and only if the $p$-torsion group schemes equipped with induced action and polarization, $(A[p], \iota, \lambda)$ and $(A'[p], \iota', \lambda')$, are isomorphic over  $\kk$. The Ekedahl-Oort strata are locally closed, and the closure of each Ekedahl-Oort stratum is a union of Ekedahl-Oort strata.

The Newton stratification is based on the isogeny class of the $p$-divisible group of the parameterized abelian varieties. Two field-valued points $(A, \iota, \lambda, \xi)$ and $(A', \iota', \lambda', \xi')$ of $\mathcal{M}(a,b)$ are in the same \textbf{Newton stratum} if and only if the $p$-divisible groups equipped with induced action and polarization, $(A[p^\infty], \iota, \lambda)$ and $(A'[p^\infty], \iota', \lambda')$, are isogenous (in a way that respects the actions and polarizations) over $\kk$. 

The Newton strata are locally closed, and the closure of each Newton stratum is a union of Newton strata. The unique closed Newton stratum of $\mathcal{M}(a,b)$ is the supersingular locus, which we denote as $\mathcal{M}(a,b)^{ss}$. In particular, a point  $(A, \iota, \lambda, \xi)$  of $\mathcal{M}(a,b)(\kk)$ is contained in the supersingular locus if and only if $A$ is a supersingular abelian variety. 

\subsection{Weyl Group Cosets}\label{sec:WeylGroupCosets}

Results of \cite{Moonen2001} relate Ekedahl-Oort strata to cosets in a certain Weyl group. This section introduces the relevant Weyl group cosets and their minimal-length coset representatives.

The Weyl group that is relevant for the study of $\mathcal{M}(a,b)$ is $W=\mathfrak{S}_q$, the symmetric group on $q = a+b$ elements. We  consider $W$ as a Coxeter group with a set of \textbf{simple reflections}
$$S = \{s_1, \dots, s_{q-1} \}, \text{where }s_i = (i, \ i+1 ).$$
 
The \textbf{length} of $w\in W$, denoted $\ell(w)$, is the length of a shortest expression for $w$ as a product of simple reflections. It is proved in \cite[Proposition~1.5.2]{bjorner2006combinatorics} that the length of an element $w \in \mathfrak{S}_q$ can be computed as the number of inversions, i.e., the cardinality of the set
\begin{align}\label{eq:SnInversions}
    \mathrm{Inv}(w):=\{(i,\, j)\,|\, i<j \text{ and } w(i)>w(j)\}.
\end{align}

In particular, $W$ has a unique element $w_0$ of maximal length, where $w_0(k) = q+1-k$. 

For $J \subset S$, let $W_J$ denote the subgroup of $W$ generated by $J$. Note that $W_J$ is a parabolic subgroup of $W$, and it follows from Proposition 2.4.4 of \cite{bjorner2006combinatorics} that every right coset of  $W_J\setminus W$ contains a unique minimal-length coset representative. Let ${}^JW$ be the collection of such minimal-length coset representatives.

For the subset $J_{(a,b)} = \{s_1, ..., s_{q-1} \} \setminus \{s_b\}$ of $S$, we let $W_{(a,b)} := W_{J_{(a,b)}}$, and $\mathbf{W}(a,b) := {}^{J_{(a,b)}}W$. The following theorem is paraphrased from Theorem 6.7 of \cite{Moonen2001}:

\begin{theorem}[Moonen]\label{thm:Moonenbij}

There is a bijection of sets:
\[\set{\text{Ekedahl-Oort Strata of $\mathcal{M}(a,b)$}} \longleftrightarrow \mathbf{W}(a,b).\]
\end{theorem}

We now recall some properties of Moonen's construction. (See also \cite{wooding_2016}.) Let $G$ be the group of $\mathcal{O}_K \otimes_{\mathbx{Z} } \mathbx{Z}_{(p)}$-linear symplectic similitudes of $\Lambda$. Then $G$ is a group scheme over $\mathbx{Z}_p$, and we let $\overline{G}$ be its special fiber. Moonen \cite{Moonen2001} gives an explicit identification of the Ekedahl-Oort strata with $W_X \backslash W_{\overline{G}}$, for a certain subgroup $W_X$ depending on the signature. Concretely, $W_X \backslash W_{\overline{G}}$ can be described as
$$\{ (w_1, w_2) \in  \mathbf{W}(a,b) \times \mathbf{W}(b,a) \ | \ w_2 = w_0 w_1 w_0 \}.$$

As the map $\pi \mapsto w_0\pi w_0$ gives an isomorphism between $\mathbf{W}(a,b)$ and $\mathbf{W}(b,a)$, one has $W_X \backslash W_{\overline{G}}\cong \mathbf{W}(a,b)$ which agrees with our statement of \cref{thm:Moonenbij}.

For any $w \in  \mathbf{W}(a,b)$, let $\mathcal{M}(a,b)_w$ denote the corresponding Ekedahl-Oort stratum and $(G_w, \iota_w, \lambda_w)$ denote the corresponding $p$-torsion group scheme. One observes that since $G_w$ is a $p$-torsion group scheme, the action of $\mathcal{O}_K$ on $G_w$, via $\iota_w$, factors through $\mathcal{O}_{K}/(p)$. Since $p$ is assumed to be inert, $\mathcal{O}_{K}/(p) = \ff_{p^2}$, and we abuse notation and refer to this induced action of $\ff_{p^2}$ as $\iota_w$ as well. Moonen uses (contravariant) Dieudonn\'e theory to describe  $(G_w, \iota_w, \lambda_w)$. For each $w \in  \mathbf{W}(a,b)$, he constructs the \textbf{standard object} $(N_w, F, V)$. This is the Dieudonn\'e  module of $G_w$, consisting of a vector space $N_w$ of dimension $2q$ over $\kk$, $F$ a $\text{Frob}_{\kk}$-semilinear operator on $N_w$, and $V$ a $\text{Frob}^{-1}_{\kk}$-semilinear operator on $N_w$, described explicitly on a basis. The action $\iota_w$ is recorded by a splitting $N_w = N_{w,1} \oplus N_{w,2}$. By Theorem 6.7 of \cite{Moonen2001}, $\lambda_w$ is uniquely determined by $(G_w, \iota_w)$, so it is unnecessary to record the corresponding polarization of $N_w$. 

We now outline the inverse of the assignment $\omega \mapsto N_{\omega}$. Given a Dieudonn\'{e} module $(N,F,V)$ with an $\ff_{p^2}$-action, one constructs a \emph{final filtration} 
$$0 \subset W_1 \subset W_2 \subset \cdots \subset W_{2q}=N$$ 
that is stable under $F$ and $V^{-1}$ and has the property $\dim_\kk(W_i)=i$. 

The action of $\ff_{p^2}$ on $N$ induces a decomposition $N = N_1 \oplus N_2$. Intersecting the filtration $W_\bullet$ of $N$ with the subspace $N_i$ (for $i=1,2$) gives:
$$0 \subset C_{i,1} \subset C_{i,2} \subset \cdots \subset C_{i,q}=N_i.$$
From here we define the functions 
$$\eta_i(j)= \dim(C_{i,j} \cap N[F]).$$
We focus on $\eta_1$, as $\eta_2$ is determined by $\eta_1$ via the Rosati involution condition. Since \[\eta_1(q)=\dim(N_1[F])=b,\]
there are $b$ indices where $\eta_1$ jumps (meaning $\eta_1(j) = \eta_1(j-1)+1$). We denote these integers by $1 \leq j_1 \leq \ldots \leq j_b \leq q$ and the remaining integers by $1 \leq i_1 \leq \ldots i_a \leq q$. The permutation $\omega_N \in \mathbf{W}(a,b)$ is
\begin{equation}\label{eqn:omegafromDD}
    \omega_N(j_l)=l \text{ and }
    \omega_N(i_m)=b+m.
\end{equation}

By construction, $N$ is the standard object corresponding to the permutation $\omega_N$.

In \cref{sec:eo}, we give more explicit descriptions in signature $(q-2,2)$ of the minimal length coset representatives of $W_J\setminus W$ and the standard objects.


\section[Structure of Ekedahl-Oort Strata for Signature (q-2,2)]{Structure of Ekedahl-Oort Strata for Signature $(q-2,2)$}\label{sec:eo}
The goal of this section is to index the Ekedahl-Oort strata of $\mathcal{M}(q-2,2)$ and study the topological closure relations between them. First, we explicitly describe the elements of the set $\textbf{W}(q-2,2)$ and recall the Bruhat partial ordering of these elements. Following that, we discuss a generalization of the Bruhat order due to He \cite{HeOrder2007} that corresponds to topological closure relations between the Ekedahl-Oort strata. Finally, we present some results on the closure relations between the elements of $\mathbf{W}(q-2,2)$.

\subsection{Index Set \& Bruhat Order}\label{sec:WeylIndexSet}

In this section, we explicitly describe the set $\mathbf{W}(a,b)$ and the Bruhat order between its elements for a general signature $(a,b)$. We end the section by summarizing our results for the case $(a,b) = (q-2,2)$.

The result in \cite[Lemma 2.4.7]{bjorner2006combinatorics} shows that $$\mathbf{W}(a,b) = \setp{\gamma \in W}{\gamma^{-1}(1) < \cdots < \gamma^{-1}(b) \text{ and } \gamma^{-1}(b + 1) < \cdots < \gamma^{-1}(a + b)}.$$ Observe that any $\gamma \in \mathbf{W}(a,b)$ is uniquely determined by $\underline{u} := (u_1,\ldots,u_b)=(\gamma^{-1}(1),\ldots,\gamma^{-1}(b))$. Therefore, by $\gamma_{\underline{u}}$,  we mean the unique permutation in $W$ such that $\gamma_{\underline{u}}(u_i) = i$ and $\gamma_{\underline{u}}(z_j) = b + j$, where $z_1 < \cdots < z_a$ is the complement of the set $\set{u_1,\ldots,u_b}$ in $\set{1,\ldots,a+b}$. Necessarily, we have
\begin{align}\label{eq:coset-rep}
    \mathbf{W}(a,b) = \setp{\gamma_{\underline{u}}}{\underline{u}  = (u_1,\ldots,u_b) \text{ such that }1\leq u_1 < \cdots < u_b \leq a + b}.
\end{align}
In particular, $\gamma_{\underline{u}}$ is an order-preserving bijection on both $\set{u_1,\ldots,u_b}$ and its complement.
\begin{lemma}\label{lemma:gammalength}
    $\ell(\gamma_{\underline{u}}) = \sum_{i=1}^b(u_i - i)$
\end{lemma}
\begin{proof}
    We compute the length of $\gamma_{\underline{u}}$ by studying the inversion set in Equation \eqref{eq:SnInversions}. Let $S = \set{1,\ldots,a+b}$. Define the set $U = \set{u_1,\ldots,u_b}$, where $1\leq u_1 < \cdots < u_b \leq a + b$, and let $Z = \set{z_1,\ldots,z_a}$, where $z_1 < \cdots < z_a$, be its complement in $S$.
    For $1 \leq x < y \leq a + b$, consider $\gamma_{\underline{u}}(x)$ and $\gamma_{\underline{u}}(y)$. Since $\gamma_{\underline{u}}$ is order preserving on $U$ and $Z$, if $x,y \in U$ or if $x,y \in Z$, then we do not have $\gamma_{\underline{u}}(x) > \gamma_{\underline{u}}(y)$. 
    Observe that $x \in U$ if and only if $1 \leq \gamma_{\underline{u}}(x) \leq b$. This necessarily gives us
    \[\mathrm{Inv}(\gamma_{\underline{u}}) = \setp{(x,y)}{x < y \text{ and } x \notin U,\, y \in U}\]
    Assume $y = u_i$. Then since $x < y$ and $x \notin U$, we obtain that $x$ has $(u_i - i)$-many choices. Therefore,
    \[\ell(\gamma_{\underline{u}}) = \#\mathrm{Inv}(\gamma_{\underline{u}}) = \sum_{i=1}^b(u_i - i).\qedhere
    \]
    \end{proof}

We can make the permutation $\gamma_{\underline{u}}$ explicit.
\begin{proposition}\label{lemma:gxt-cosetreps}
$\gamma_{\underline{u}} = (b, b+1, \ldots ,u_b) \cdots (2, 3, \ldots ,u_2)(1,2,\ldots,u_1)$. 
\end{proposition}
\begin{proof}
    Let $S,U,$ and $Z$ be the sets defined in the proof of Lemma \ref{lemma:gammalength}. The permutation $\gamma_{\underline{u}}$ is defined by the property $\gamma_{\underline{u}}(u_i) = i$ and $\gamma_{\underline{u}}(z_j) = b + j$. Let $\sigma_{\underline{u}} = \sigma_b\cdots \sigma_2\sigma_1$ where $\sigma_i = (i,i+1,\ldots,u_i)$. We clearly see that $\sigma_{\underline{u}}(u_i) = i$. In particular, $\sigma_{\underline{u}}$ is an order-preserving bijection from $U$ to $\set{1,\ldots,b}$. Therefore, necessarily, it is a bijection from $Z$ to $\set{b+1,\ldots,a+b}$. It is sufficient to prove that $\sigma_{\underline{u}}$ is order-preserving on $Z$ to conclude $\sigma_{\underline{u}} = \gamma_{\underline{u}}$, as we will have shown that one necessarily has $\sigma_{\underline{u}}(z_j) = b + j$.

    For any $z \in Z$, let $\prescript{(0)}{}{z} = z$ and inductively define $\prescript{(k)}{}{z} = \sigma_k(\prescript{(k-1)}{}{z})$; necessarily, $\prescript{(b)}{}{z} = \sigma_{\underline{u}}(z)$.
    
    We prove inductively that for $z,z' \in Z$, if $z < z'$ then $\prescript{(k)}{}{z} < \prescript{(k)}{}{z}'$, for $0 \leq k \leq b$. By assumption, the statement for $k = 0$ is true. Therefore, assume that $\prescript{(k-1)}{}{z} < \prescript{(k-1)}{}{z}'$ for some $1 \leq k \leq b$, and consider $\prescript{(k)}{}{z} = \sigma_k(\prescript{(k-1)}{}{z})$ and $\prescript{(k)}{}{z}' = \sigma_k(\prescript{(k-1)}{}{z}')$. Since $\sigma_k = (k,k+1,\ldots,u_k)$, this permutation will invert the pair $(\prescript{(k-1)}{}{z},\prescript{(k-1)}{}{z}')$ only if $\prescript{(k-1)}{}{z}' = u_k$ and $k \leq \prescript{(k-1)}{}{z} < u_k$. 
    
    Suppose $\prescript{(k-1)}{}{z}' = u_k$. Then $\sigma_k(\prescript{(k-1)}{}{z}') = k$ and, therefore,
    \[\sigma_{\underline{u}}(z') = \sigma_b\cdots\sigma_k(\prescript{(k-1)}{}{z}') = k\]
    since $\sigma_j(k) = k$, for any $j > k$. This is a contradiction since one must have $b + 1 \leq \sigma_{\underline{u}}(z') \leq a + b$ but $1 \leq k \leq b$. Thus, $\prescript{(k-1)}{}{z}' \neq u_k$ and $\sigma_k$ does not invert the pair $(\prescript{(k-1)}{}{z},\prescript{(k-1)}{}{z}')$, i.e., $\prescript{(k)}{}{z} \leq \prescript{(k)}{}{z}'$. We obtain a strict inequality since $\sigma_k$ is a bijection. By induction, we have $\sigma_{\underline{u}}(z) = \prescript{(b)}{}{z} < \prescript{(b)}{}{z}' = \sigma_{\underline{u}}(z')$. Hence, $\sigma_{\underline{u}}$ is an order-preserving bijection on $Z$.
\end{proof}

\begin{definition}[The Bruhat order]\label{def:BruhatDefSequence}
Consider the Coxeter system $(W,S)$. For $w$ and $w'$ in $W$, we say $w' \leq w$, if $\ell(w')\leq \ell(w)$ and there exists a sequence $w' = v_0,\,v_1,\,\ldots,\,v_m = w$ such that
\begin{itemize}
\item $\ell(v_{i-1}) \leq \ell(v_i)$; and
\item $v_{i-1}^{-1}v_i$ is a reflection.
\end{itemize}
\end{definition}

By  \cite[Proposition~2.5.1]{bjorner2006combinatorics}, the Bruhat order on $W$ induces a partial order on $\mathbf{W}(a,b)$, also called the Bruhat order and denoted $\leq$. This can described explicitly, and we record this as the following proposition.

\begin{proposition}\label{prop:level1Bruhat}
    For $\gamma_{\underline{u}},\gamma_{\underline{u}'} \in \mathbf{W}(a,b)$, we have $\gamma_{\underline{u}} \leq \gamma_{\underline{u}'}$ if and only if $u_i \leq u_i'$ for $1 \leq i \leq b$.
\end{proposition}
\begin{proof}
    See \cite[Proposition~2.4.8]{bjorner2006combinatorics}.
\end{proof}

\begin{remark}\label{remark:gxt-EOStrata}
Each $\gamma_{\underline{u}} \in \mathbf{W}(a,b)$ corresponds to an Ekedahl-Oort stratum labeled $\mathcal{M}(a,b)_{\gamma_{\underline{u}}}$. By \cite{Moonen2001}, the dimension of $\mathcal{M}(a,b)_{\gamma_{\underline{u}}}$ is equal to $\ell(\gamma_{\underline{u}})$, which was shown to equal $\sum_{i=1}^b(u_i - i)$ in Lemma \ref{lemma:gammalength}. Let $n_d(a,b)$ be the number of Ekedahl-Oort strata there are of dimension $0 \leq d \leq ab$, i.e.,
\[n_d(a,b) = \#\setp{\gamma_{\underline{u}}}{\ell(\gamma_{\underline{u}}) = d}.\]
This counts the number of tuples $\underline{u}$ such that $\sum_{i=1}^b (u_i - i) = d$. Letting $x_i = u_i - i$, this is equivalent to determining the integers $0 \leq x_1 \leq \cdots \leq x_b \leq a$ for which we have $x_1 + \cdots + x_b = d$. That is, $n_d(a,b)$ is counting the number of partitions of $d$ with at most $b$ parts with each part of size at most $a$. This is famously computed as the $\mathtt{q}^d$ coefficient of the Gaussian binomial coefficient
\[\binom{a + b}{b}_{\mathsf{q}} = \prod_{i=0}^{b-1} \frac{1-\mathtt{q}^{(a+b)-i}}{1 - \mathtt{q}^{i+1}}\]
where $\mathtt{q}$ is a formal variable.
\end{remark}

We specialize our descriptions to the case $(a,b) = (q-2,2)$ below.
\begin{corollary}\label{prop:EO-comma-2}
    We have $\mathbf{W}(q-2,2) = \setp{\gamma_{u,v}}{1\leq u < v \leq q}$, where $\gamma_{u,v} = (2,3,\ldots,v)(1,2,\ldots,u)$ and $\ell(\gamma_{u,v}) = u + v - 3$. Elements of $\mathbf{W}(q-2,2)$ satisfy $\gamma_{u,v} \leq \gamma_{u',v'}$ if and only if $u \leq u'$ and $v\leq v'$.
    
    Furthermore,
    \[n_d(q-2,2) = \begin{cases} {\displaystyle \lfloor d/2 \rfloor + 1} & \text{if }d \leq q - 2,\\[0.5em] {\displaystyle \max(0,\lfloor d/2 \rfloor + 1} - (d - (q - 2))) & \text{if }d>q-2. \end{cases}\]
\end{corollary}

\subsection{Standard Objects}\label{sec:gxt-standardobjects}

In this section, we describe the standard objects of the Ekedahl-Oort strata of $\mathcal{M}(q-2,2)$. These standard objects are mod-$p$ Dieudonn\'{e} modules, and by Dieudonn\'{e} theory each determines a corresponding $p$-torsion group scheme.
\begin{lemma} \label{lem:SOuv}
Let $(N_{\gamma_{u,v}},F,V)$ be the standard object corresponding to $\gamma_{u,v}$. Then,
$$N_{\gamma_{u,v}} = \mathrm{Span}_\kk\{ e_{i,j} \},$$
where $1 \leq i \leq 2$ and $1 \leq j \leq q$. Further, $F$ is the $\mathrm{Frob}_{\kk}$-semilinear operator on $N_{\gamma_{u,v}}$ and $V$ is the $\mathrm{Frob}^{-1}_{\kk}$-semilinear operator on $N_{\gamma_{u,v}}$ extended from the following action on the basis:
\begin{align*}
F(e_{1,j}) &= \begin{cases} 
    0  &\text{if $j=u,\,j=v$,} \\ 
    e_{2,j} &\text{if $1 \leq j \leq u-1$,} \\  
    e_{2,j-1} &\text{if $u < j \leq v-1$,} \\  
    e_{2,j-2} &\text{if $j>v$,} \end{cases}
&
V(e_{1,j}) &= \begin{cases} 
    0 &\text{if $j=1,\,2$,}  \\ 
    e_{2,j-2} &\text{if $2 < j \leq q-v+2$,} \\ 
    e_{2,j-1} &\text{if $q-v+2 <j \leq q-u+1 $,} \\ 
    e_{2,j} &\text{if $j>q-u+1$,} \end{cases} \\[0.5em]
F(e_{2,j}) &= \begin{cases} 
    e_{1,1} &\text{if $j= q-v+1$,}  \\ 
    e_{1,2} &\text{if $j=q-u+1$,} \\ 
    0 &\text{otherwise,} \end{cases} &
V(e_{2,j}) &= \begin{cases} 
    e_{1,u} &\text{if $j=q-1$,} \\ 
    e_{1,v} &\text{if $j=q$,}  \\ 
    0 &\text{otherwise}. \end{cases}
\end{align*}
\end{lemma}
\begin{proof}
We use the explicit description of $\gamma_{u,v}$ to unwind the description of the corresponding standard object given in \cite[4.9]{Moonen2001}, under the assumption that the prime $p$ is inert in the imaginary quadratic field $K$.
\end{proof}

\subsection{The Closure Order}\label{sec:ClosureOrder}

\subsubsection{Definition and properties}\label{sec:ClosureOrderDef}

There are two partial orders on $\mathbf{W}(q-2,2)$ that we are concerned with in this paper: the Bruhat order defined above, and a generalization of this due to He \cite{HeOrder2007} that we refer to as the Closure order. In this section we define the Closure order and discuss how both orders relate to Ekedahl-Oort strata. We focus on $\mathbf{W}(q-2,2)$ but the constructions here can all be re-framed for more general $\mathbf{W}(a,b)$. We refer to \cite{HeOrder2007,  PinkWedhornZiegler2011, ViehmannWedhorn2013,wedhorn2005specialization, wooding_2016} for more information.

Recall that the Frobenius automorphism $F$ of the algebraic group $\overline{G}$ induces a $\psi \in \mathrm{Aut}(W,S)$ of the Coxeter system: for any element $w\in W$, $\psi(w)=w_0ww_0$, where $w_0$ is the unique maximal element in $W$. 

As in \cref{sec:WeylGroupCosets}, $W_{(q-2,2)}$ denotes the subgroup of the Weyl group $W$ generated by $J_{(a,b)}=\{s_1, ..., s_{q-1} \} \setminus \{s_b\}$. For ease of notation, let $J:=J_{(q-2,2)}$. The unique maximal element in $W_{(q-2,2)}$, with respect to the Bruhat order, is denoted by $w_{0,J}$. Let $W^{\psi(J)}$ be the set of elements $w\in W$ that have minimal length in their coset $wW_{\psi(J)}$, and let $J':={}^{w_0}\psi(J)$ denote the set $w_0 \psi(J) w_0$. In the notation of \cref{sec:WeylGroupCosets}, ${}^{J'}W^{\psi(J)}$ denotes the collection of minimal-length coset representatives for $W_{J'}\big\backslash W\big\slash\psi(J)$.

The definition of the Closure order involves the element $x$ of minimal length in the double coset $W_{J'} w_0 W_{\psi(J)}$. One can show that this element satisfies
$$x=w_{0,J'} w_0 = w_0 w_{0,\psi(J)},$$
and so $x$ is the unique maximal length element amongst the minimal length coset representatives ${}^{J'} W^{\psi(J)}$. 

We now define the \textbf{Closure order} corresponding to $\psi$ on $\mathbf{W}(q-2,2)$.

\begin{definition}[The Closure order $\Heleq$]\label{def:HeOrderDef}
  Let $w,w'$ be elements of $\mathbf{W}(q-2,2)$. Then $w\Heleq w'$ if there exists an $h\in W_{(q-2,2)}$ such that 
\begin{equation}\label{eq:HeOrderDef}
    hwx\psi(h^{-1})x^{-1}\leq w',
\end{equation}
where $\leq$ denotes the Bruhat order\footnote{In the notation of \cite{HeOrder2007}, Inequality \eqref{eq:HeOrderDef} actually implies $wx\Heleq w'x$. In our work, as in \cite{ViehmannWedhorn2013}, we drop the $x$ from the notation and simply write  $w\Heleq w'$.}.
\end{definition}

Using the fact that $\psi(h)=w_0hw_0$, we have
$$x\psi(h^{-1})x^{-1}=(w_0w_{0,\psi(J)})(w_0{ h^{-1}}w_0)(w_{0,\psi(J)}w_0)= w_{0,J}h^{-1}w_{0,J},$$ and so Inequality \eqref{eq:HeOrderDef} can be rewritten as 
\begin{equation}\label{eq:HeOrderDef2}
    hww_{0,J}h^{-1}w_{0,J}  \leq w'.
\end{equation}
Note that the Bruhat order is the special case of the Closure order where Inequality \eqref{eq:HeOrderDef} is satisfied for $h$ equal to the identity in $W_{(q-2,2)}$.

In \cref{sec:ClosureOrder} we give some partial results on the Closure order on $\mathbf{W}(q-2,2)$. One difficulty that arises in giving a complete description of the Closure order on $\mathbf{W}(q-2,2)$ is that the expression $hww_{0,J}h^{-1}w_{0,J}$ is not a coset representative in $\mathbf{W}(q-2,2)$. In other words, the Bruhat relation that we are checking in Inequality \eqref{eq:HeOrderDef} is in the Weyl group $W$ and not in $\mathbf{W}(q-2,2)$, and so we are not able to apply our results on the Bruhat order in \cref{prop:level1Bruhat} to the Closure order.

The Closure order is of interest because it captures topological closure relations among the Ekedahl-Oort strata. Let $\overline{\mathcal{M}(q-2,2)}_{\gamma_{u',v'}}$ denote the topological closure of the stratum $\mathcal{M}(q-2,2)_{\gamma_{u',v'}}$. By \cite[Theorem~1.2]{ViehmannWedhorn2013},
$$\overline{\mathcal{M}(q-2,2)}_{\gamma_{u',v'}} = \!\!\!\!\!\!\bigcup_{\gamma_{u,v} \Heleq \gamma_{u',v'} } \!\!\!\!\!\!\mathcal{M}(q-2,2)_{\gamma_{{u,v}}}.$$

Thus, to understand closure relations among Ekedahl-Oort strata, it suffices to analyze the Closure order relations in $\mathbf{W}(q-2,2)$.

\subsubsection{Closure relations}\label{sec:HeRelations}

In this section, we prove several results that give a partial classification of the Closure order for signature $(q-2,2)$. We begin with a lemma that we will use in order to prove the two sets of closure relations in Theorems \ref{thm:primaryHerelations} and \ref{thm:secondaryHerelations}.
\begin{lemma}\label{lemma:s_k_identities}
    Let $s_k$ with $k>2$ denote the simple reflection $(k,k+1)$. For $1\leq u<v\leq q$, the following hold: 
    \begin{itemize}
        \item $\gamma_{u,v+1}s_v=\gamma_{u,v}$ \text{ and }  $ s_v\gamma^{-1}_{u,v+1}=\gamma^{-1}_{u,v};$
        \item $\gamma_{u+1,v}s_u=\gamma_{u,v}$ \text{ and }  $s_u\gamma^{-1}_{u+1,v}=\gamma^{-1}_{u,v}$.
    \end{itemize}
    Furthermore, $s_{q-(j-1)} s_{q-(j-2)}\cdots s_{q-(j-i)}\gamma^{-1}_{j+1,q+1-j+i}=\gamma^{-1}_{j+1,q+1-j}\,$ for any $4\leq j \leq \frac{q}{2}$ and $1\leq i\leq j-3$. 
\end{lemma}

\begin{proof}
    We can prove these results using the non-disjoint cycle definition of $\gamma_{u,v}$:
    $$\gamma_{u,v}=(2,3,\ldots, v-1, v)(1,2,\ldots,u-1,u).$$

    For example, to prove the first equality, note that since $v>u$, the reflection $s_v$ commutes with $(1,2,\ldots,u-1,u)$. Hence,
    \begin{align*}
        \gamma_{u,v+1}s_v&=(2,3,\ldots, v, v+1)(1,2,\ldots,u-1,u)(v,v+1)\\
        &=(2,3,\ldots, v, v+1)(v,v+1)(1,2,\ldots,u-1,u)\\
        &= (2,3,\ldots, v)(1,2,\ldots,u-1,u)
    \end{align*}
    The other relations in the bullet-points are proved in a similar manner, and the last result follows from applying these repeatedly.
\end{proof}

We define an action on elements $w\in\mathbf{W}(q-2,2)$ by elements $h\in W_{(q-2,2)}$ by
\begin{equation}\label{eqn:He_action}
    h\bullet w:=  hw w_{0,J}h^{-1}w_{0,J},
\end{equation}
which is the expression that appears on the left side of Inequality \ref{eq:HeOrderDef2}. In other words, $\gamma_{u_1,v_1}\Heleq \gamma_{u_2,v_2}$  if there is an $h\in W_{(q-2,2)}$ such that the Bruhat relation $h\bullet\gamma_{u_1,v_1}\leq \gamma_{u_2,v_2}$ holds. In our work, we express $h$ as a product of simple reflections, and the following lemma will be used throughout.
\begin{lemma}\label{lemma:w_0s_kw_0}
    For any simple reflection $s_k=(k, k+1)$ with $k>2$, $$w_{0,J}s_kw_{0,J}=s_{q+2-k}.$$
\end{lemma}
\begin{proof}
Since $w_{0,J}$, defined as $w_{0,J}(x) = q + 3 - x$ for all $3 \leq x \leq q$, is of order two and $k > 2$, we have
\[w_{0,J}s_kw_{0,J} = (w_{0,J}(k),w_{0,J}(k+1)) = s_{q+2 - k}. \qedhere\]
\end{proof}

The following theorem gives the first set of closure relations.
\begin{theorem}\label{thm:primaryHerelations}
Let $q\geq 5$. Then the following relations hold
\begin{enumerate}
    \item $\gamma_{j+1,q+1-j}\Heleq\gamma_{j,q+3-j}\,$ for $3\leq j < q/2$, 
    \item $\gamma_{q-j,j+1}\Heleq\gamma_{q-j+2,j}\,$ for $q/2 + 1< j\leq q-1$.
\end{enumerate}
\end{theorem} 

\begin{remark}
For both relations, the difference in lengths between the coset representatives being compared is 1. We restrict to $q\geq 5$ since the specific $\gamma_{u,v}$ in the statement of the theorem are not defined for $q\leq 4$. For smaller values of $q$, there are no closure relations other than those coming from Bruhat relations. 
\end{remark}

In many computations, it is easier to work with the inverses $\gamma_{u,v}^{-1}$ of the coset representatives. To prove the above result, \cref{thm:primaryHerelations}, we will demonstrate that \[s_j\bullet \gamma^{-1}_{j+1,q+1-j}\leq \gamma^{-1}_{j, q+3-j} \quad \text{and} \quad s_j\bullet \gamma^{-1}_{q-j,j+1}\leq \gamma^{-1}_{q+2-j,j}\] with respect to the Bruhat order. There are two conditions to check for each of the relations in \cref{thm:primaryHerelations}, and we split these into separate lemmata. We first compute the action by $s_j$ on $\gamma^{-1}_{j+1,q+1-j}$ and $\gamma^{-1}_{q-j,j+1}$.

\begin{lemma}\label{lemma:primarytauj1}
Let $3\leq j<q/2$, and let $\tau_{j}$ be defined by 
\begin{equation}\label{eq:tauj1}
    \tau_{j}(k)=
    \begin{cases}
    j& \text{for $k=1$,}\\
    q+1-j& \text{for $k=2$,}\\
    k-2 & \text{for $3 \leq k\leq j+1$,}\\ 
    k-1 & \text{for $j+2 \leq k\leq q+1-j$,}\\ 
    q+3-j& \text{for $k=q+2-j$,}\\
    q+2-j& \text{for $k=q+3-j$,}\\
    k& \text{for $k>q+3-j$.}
    \end{cases}
\end{equation}
Then $s_j\bullet \gamma^{-1}_{j+1,q+1-j}=\tau_{j}$.
\end{lemma}

\begin{proof}
    Applying \cref{lemma:s_k_identities} and \cref{lemma:w_0s_kw_0} yields
     $$s_j\bullet \gamma^{-1}_{j+1,q+1-j}= s_j \gamma^{-1}_{j+1,q+1-j} s_{q+2-j} =\gamma^{-1}_{j,q+1-j} s_{q+2-j}.$$
Recall that $\gamma_{u,v}=(2,3,\ldots, v)(1,2,\ldots, u)$. Thus, the expression above can be written as
$$(1,2,\ldots, j)^{-1}(2,3,\ldots, q+1-j)^{-1} s_{q+2-j}.$$

We now determine the image of each $k$ under the above permutation, focusing first on small values of $k$. Moving right to left through the permutations, we see that 1 is mapped as: $1\mapsto j$. The element $2$ is mapped as: $2\mapsto q+1-j$. For all other $k\leq j+1$, we have $k\mapsto k-1\mapsto k-2$ via $(1,2,\ldots, j)^{-1}(2,3,\ldots, q+1-j)^{-1}$.

For $j+2\leq k\leq q+1-j$,  we have $k\mapsto k-1$ via $(2,3,\ldots, q+1-j)^{-1}$. Since $s_{q+2-j}$ is disjoint from the rest of the permutations, $q+2-j$ and $q+3-j$ are swapped. Finally, all $k>q+3-j$ are fixed.
\end{proof}

\begin{lemma}
Let $q/2+1< j\leq q-1$ , and let $\tau_{j}$ be defined by 
\begin{equation}\label{eq:tauj2}
    \tau_{j}(k)=
    \begin{cases}
    q-j& \text{for $k=1$,}\\
    j& \text{for $k=2$,}\\
    k-2 & \text{for $3 \leq k\leq j+1$,}\\ 
    q+2-j& \text{for $k=q+2-j$,}\\
    q+1-j& \text{for $k=q+3-j$,}\\
    k-1 & \text{for $q+4-j\leq k\leq j$,}\\
    k& \text{for $k>q+3-j$}.
    \end{cases}
\end{equation}
Then $s_j\bullet \gamma^{-1}_{q-j,j+1}=\tau_{j}$.
\end{lemma}

\begin{proof}
    Applying \cref{lemma:s_k_identities} and \cref{lemma:w_0s_kw_0}  yields
\[ s_j\bullet  \gamma^{-1}_{q-j,j+1}= s_j \gamma^{-1}_{q-j,j+1}s_{q+2-j}
         =\gamma^{-1}_{q-j,j} s_{q+2-j},\]
which can be written as
$$(1,2,\ldots, q-j)^{-1}(2,3,\ldots, j)^{-1} s_{q+2-j}.$$

The result then follows using techniques similar to those used in the proof of \cref{lemma:primarytauj1}.
\end{proof}

We now prove the following result for the lengths of the $\tau_j$'s.

\begin{lemma}     
If $3\leq j < q/2$, then $\ell(s_j\bullet \gamma^{-1}_{j+1,q+1-j}) = \ell(\gamma^{-1}_{j+1,q+1-j})$. If $q/2 +1 < j\leq q-1$, then\\{$\ell(s_j\bullet \gamma^{-1}_{q-j,j+1}) = \ell(\gamma^{-1}_{q-j, j+1})$}.
\end{lemma}
 \begin{proof} 
Recall that one may compute the length of a permutation as the cardinality of its set of inversions (Equation \ref{eq:SnInversions}).  We obtain the result by a straightforward computation of the relevant sets of inversions.
 \end{proof} 
 
The following is the final ingredient needed in order to prove \cref{thm:primaryHerelations}. 
\begin{lemma}\label{lemma:primaryHerelations}
If $3\leq j< q/2$  then $(s_j\bullet \gamma^{-1}_{j+1,q+1-j})^{-1}\gamma^{-1}_{j, q+3-j} = (2,q+2-j)$. If $q/2 +1< j\leq q-1$  then   $(s_j\bullet \gamma^{-1}_{q-j,j+1})^{-1}\gamma^{-1}_{q+2-j,j}=(1,q+2-j).$
\end{lemma}
\begin{proof}
The proofs of these equalities use the above lemmata and simple permutation composition. For the first equality, we use the definition of $\tau_j$ in Equation \eqref{eq:tauj1} to write
\begin{align*}
	(s_j\bullet \gamma^{-1}_{j+1,q+1-j})^{-1}\gamma^{-1}_{j, q+3-j} &=\tau_j^{-1}(1,2,\ldots,j)^{-1}(2,3,\ldots,q+3-j)^{-1}. 	
\end{align*}
We verify that this simplifies to $(2,q+2-j)$ by composition of permutations. For example, $\gamma^{-1}_{j, q+3-j}$ maps $2\mapsto q+3-j$, which then maps to $q+2-j$ via $\tau_j^{-1}$. Similarly, $\gamma^{-1}_{j, q+3-j}$ maps $q+2-j \mapsto q+1-j$, which is then mapped to $2$ via $\tau_j^{-1}$. All other values are fixed by the composition.

For the second equality, we use the definition of $\tau_j$ in Equation \eqref{eq:tauj2} to write
\begin{align*}
(s_j\bullet \gamma^{-1}_{q-j,j+1})^{-1}\gamma^{-1}_{q+2-j,j}&=\tau_j^{-1}(1,2,\ldots,q+2-j)^{-1}(2,3,\ldots,j)^{-1}. 
\end{align*}
We can use the same method as above to prove the desired result. 
\end{proof}
 
These results combine to give a proof of \cref{thm:primaryHerelations}.
 \begin{proof}[Proof of \cref{thm:primaryHerelations}]
 
 The above lemmata demonstrate that $s_j\bullet \gamma^{-1}_{j+1,q+1-j}\leq \gamma^{-1}_{j, q+3-j}$ and $s_j\bullet \gamma^{-1}_{q-j,j+1}\leq \gamma^{-1}_{q+2-j,j}$, for the appropriately chosen values of $j$, with respect to the Bruhat order, which proves the two closure relation statements.
 \end{proof}

The following theorem gives the second set of closure relations.
\begin{theorem}\label{thm:secondaryHerelations}
Let $q\geq 7$. Then the following relations hold
\begin{enumerate}
    \item Let $4\leq j<q/2$ and $1\leq i\leq j-3$. Then $\gamma_{j+1,q+1-j+i}\Heleq\gamma_{j,q+3-j+i}$.
    \item Let $q/2+2< j\leq q-1$ and $1\leq i< j-q/2-1$. Then $\gamma_{q-j,j+1-i}\Heleq\gamma_{q+2-j,j-i}$.
\end{enumerate}
\end{theorem} 
\begin{remark}
As in \cref{thm:primaryHerelations}, the difference in lengths between the coset representatives being compared is 1.  We again restrict to the values of $q$, this time to $q\geq 7$, since the specific $\gamma_{u,v}$ in the statement of the theorem are not defined for smaller $q$. Note that letting $i=0$ yields the two relations given in \cref{thm:primaryHerelations}; however we need to act by different permutations in order to realize the relations.
\end{remark}

As we did for \cref{thm:primaryHerelations}, we will prove the result, \cref{thm:secondaryHerelations}, by working with the inverses of the coset representatives and by proving the corresponding Bruhat relations.  There are two conditions to check for each of the relations in \cref{thm:secondaryHerelations}, and we split these into separate lemmata. We first compute the action by certain permutations $h_{i,j}$ (defined below) on $\gamma^{-1}_{j+1,q+1-j+i}$ and $\gamma^{-1}_{q-j-i,j+1-i}$.

\begin{lemma}\label{lemma:tauij1}
Let $4\leq j<q/2$ and $1\leq i\leq j-3$. Define $h_{i,j}$ to be the permutation
\begin{equation}\label{eqn:hij_smallj}
h_{i,j}=s_js_{j-1}\cdots s_{j-i}\cdot s_{q-(j-1)} s_{q-(j-2)}\cdots s_{q-(j-i)}
\end{equation}
and let $\tau_{i,j}$ be defined by 
\begin{equation}\label{eq:tauij}
    \tau_{i,j}(k)=
    \begin{cases}
    j& \text{for $k=1$,}\\
    q+1-j& \text{for $k=2$,}\\
    k-2 & \text{for $3 \leq k\leq j+1$,}\\ 
    k-1 & \text{for $j+2 \leq k\leq q+1-j$,}\\ 
    q+3-j+i& \text{for $k=q+2-j$,}\\
    k-1 & \text{for $q+3-j \leq k\leq q+3-j+i$,}\\ 
    k & \text{for $k> q+3-j+i$}. 
    \end{cases}
\end{equation}
Then $h_{i,j}\bullet \gamma^{-1}_{j+1,q+1-j+i}=\tau_{i,j}$.
\end{lemma}

\begin{proof}
The proof of this result relies on computing permutation compositions. We provide some simplifications of the expression to make the computations easier. First, use \cref{lemma:s_k_identities} to write 
\begin{align*}
    h_{i,j}\bullet \gamma^{-1}_{j+1,q+1-j+i}&=s_js_{j-1}\cdots s_{j-i}\gamma^{-1}_{j+1,q+1-j}(w_{0,J}h_{i,j}w_{0,J})^{-1},
\end{align*}
where $(w_{0,J}h_{i,j}w_{0,J})^{-1}=s_{j+2-i}\cdots s_{j}s_{j+1}s_{q+2-(j-i)}\cdots s_{q+2-(j-1)}s_{q+2-j}.$ 

Recall that $\gamma_{u,v}=(2,3,\ldots, v)(1,2,\ldots, u)$. Thus, the expression above can be written as
$$s_js_{j-1}\cdots s_{j-i}(1,2,\ldots, j+1)^{-1}(2,3,\ldots, q+1-j)^{-1}(w_{0,J}h_{i,j}w_{0,J})^{-1}.$$
The result follows from determining the image of each $k$ in the above permutation as in the proof of Lemma \ref{lemma:primarytauj1}.
\end{proof}

\begin{lemma}\label{lemma:tauij2}
Let $q/2+2< j\leq q-1$ and $1\leq i< j-q/2-1$. Define $h_{i,j}$ to be the permutation
\begin{equation}\label{eqn:hij_bigj}
h_{i,j}=s_js_{j-1}\cdots s_{j-i}\cdot s_{q+1-(j-1)} s_{q+1-(j-2)}\cdots s_{q+1-(j-i)},
\end{equation}
 and let $\tau_{i,j}$ be defined by 
\begin{equation}\label{eq:tauij2}
    \tau_{i,j}(k)=
    \begin{cases}
    q-j& \text{for $k=1$,}\\
    j-i& \text{for $k=2$,}\\
    k-2 & \text{for $3 \leq k\leq q+1-j$,}\\  
    q+2-j+i& \text{for $k=q+2-j$,}\\ 
    q+1-j+i& \text{for $k=q+3-j$,}\\
    k-1 & \text{for $q+4-j \leq k\leq j-i$,}\\ 
    k& \text{for $k>j-i$}.
    \end{cases}
\end{equation}
Then $h_{i,j}\bullet \gamma^{-1}_{q-j,j+1-i}=\tau_{i,j}$.
\end{lemma}

\begin{proof}
We omit the proof of the result since it is similar to the proof \cref{lemma:tauij1}.
\end{proof}

\cref{lemma:tauij1} and \cref{lemma:tauij2} lead to the following corollary.
\begin{corollary}    
If $4\leq j<q/2$ and $1\leq i\leq j-3$, then 
\[(h_{i,j}\bullet \gamma^{-1}_{j+1,q+1-j+i})^{-1} \gamma^{-1}_{j,q+3-j+i} =(2,q+2-j)\] 
for $h_{i,j}$ defined in Equation \eqref{eqn:hij_smallj}. If $q/2+2<j\leq q-1$ and $1\leq i< j-q/2-1$ then \[(h_{i,j}\bullet \gamma^{-1}_{q-j,j+1-i})^{-1} \gamma^{-1}_{q+2-j,j-i} =(1,q+2-j)\] for $h_{i,j}$ defined in Equation \eqref{eqn:hij_bigj}.
\end{corollary}

We have the following result for the lengths of the  elements $\tau_{i,j}$. 
\begin{lemma}
If $4\leq j<q/2$ and $1\leq i\leq j-3$, then $\ell(h_{i,j}\bullet \gamma^{-1}_{j+1,q+1-j+i})=\ell( \gamma^{-1}_{j+1,q+1-j+i})$ for $h_{i,j}$ defined in Equation \eqref{eqn:hij_smallj}. If $q/2+2< j\leq q-1$ and $1\leq i< j-q/2-1$ then \[\ell(h_{i,j}\bullet \gamma^{-1}_{q-j,j+1-i})=\ell( \gamma^{-1}_{q-j,j+1-i})\] for $h_{i,j}$ defined in Equation \eqref{eqn:hij_bigj}.
\end{lemma}

 \begin{proof}
 Recall that one may compute the length of a permutation as the cardinality of its set of inversions (see Equation \eqref{eq:SnInversions}). We obtain the result by a straightforward computation of the relevant sets of inversions.
  \end{proof}

These results combine to give a proof of \cref{thm:secondaryHerelations}. 
 \begin{proof}[Proof of \cref{thm:secondaryHerelations}]
The above lemmata demonstrate that, for the appropriate values of $i$ and $j$ and for the appropriate $h_{i,j}$ in Equations \eqref{eqn:hij_smallj} and \eqref{eqn:hij_bigj}, we have $h_{i,j}\bullet \gamma^{-1}_{j+1,q+1-j+i}\leq\gamma^{-1}_{j,q+3-j+i}$ and $h_{i,j}\bullet \gamma^{-1}_{q-j,j+1-i}\leq\gamma^{-1}_{q+2-j,j-i}$. This proves the two closure relation statements.
 \end{proof}

\subsubsection[Example: q=11]{Example: $q=11$}

We demonstrate the above theorems with the following example, where $q=11$. Each node $(u,v)$ of the diagram in Figure \ref{figure:Closure11} represents a coset representative $\gamma_{u,v}$. We draw an arrow from $(u_2,v_2)$ to $(u_1, v_1)$ if $\gamma_{u_1,v_1}\Heleq \gamma_{u_2,v_2}$. The vertical position of a node is determined by the length of the coset representative, descending from $\ell= 18$ to $\ell=0$. Note that in the middle of the diagram, there is an ``equator'' where the organization of nodes shifts for ease of reading.

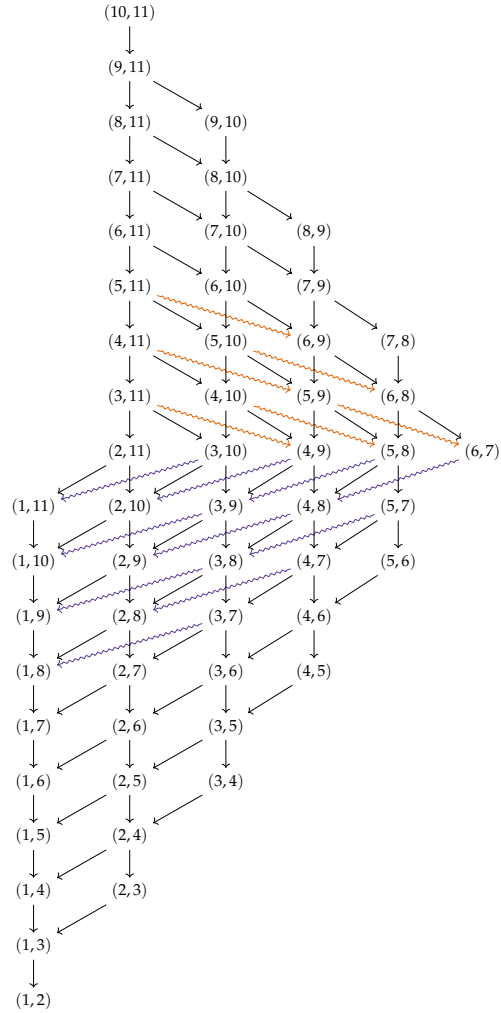
\begin{figure}[h]
{\tiny 
\[\begin{tikzcd}[cramped]
	& {(10,11)} \\
	& {(9,11)} \\
	& {(8,11)} & {(9,10)} \\
	& {(7,11)} & {(8,10)} \\
	& {(6,11)} & {(7,10)} & {(8,9)} \\
	& {(5,11)} & {(6,10)} & {(7,9)} \\
	& {(4,11)} & {(5,10)} & {(6,9)} & {(7,8)} \\
	& {(3,11)} & {(4,10)} & {(5,9)} & {(6,8)} \\
	& {(2,11)} & {(3,10)} & {(4,9)} & {(5,8)} & {(6,7)} \\
	{(1,11)} & {(2,10)} & {(3,9)} & {(4,8)} & {(5,7)} \\
	{(1,10)} & {(2,9)} & {(3,8)} & {(4,7)} & {(5,6)} \\
	{(1,9)} & {(2,8)} & {(3,7)} & {(4,6)} \\
	{(1,8)} & {(2,7)} & {(3,6)} & {(4,5)} \\
	{(1,7)} & {(2,6)} & {(3,5)} \\
	{(1,6)} & {(2,5)} & {(3,4)} \\
	{(1,5)} & {(2,4)} \\
	{(1,4)} & {(2,3)} \\
	{(1,3)} \\
	{(1,2)}
	\arrow[draw={rgb,255:red,230;green,97;blue,0}, squiggly, from=8-2, to=9-4]
	\arrow[draw={rgb,255:red,230;green,97;blue,0}, squiggly, from=8-3, to=9-5]
	\arrow[draw={rgb,255:red,230;green,97;blue,0}, squiggly, from=8-4, to=9-6]
	\arrow[draw={rgb,255:red,93;green,58;blue,155}, squiggly, from=9-3, to=10-1]
	\arrow[draw={rgb,255:red,93;green,58;blue,155}, squiggly, from=9-4, to=10-2]
	\arrow[draw={rgb,255:red,93;green,58;blue,155}, squiggly, from=9-5, to=10-3]
	\arrow[draw={rgb,255:red,93;green,58;blue,155}, squiggly, from=10-3, to=11-1]
	\arrow[draw={rgb,255:red,93;green,58;blue,155}, squiggly, from=10-4, to=11-2]
	\arrow[draw={rgb,255:red,93;green,58;blue,155}, squiggly, from=10-5, to=11-3]
	\arrow[draw={rgb,255:red,93;green,58;blue,155}, squiggly, from=11-3, to=12-1]
	\arrow[draw={rgb,255:red,93;green,58;blue,155}, squiggly, from=11-4, to=12-2]
	\arrow[draw={rgb,255:red,93;green,58;blue,155}, squiggly, from=12-3, to=13-1]
	\arrow[draw={rgb,255:red,230;green,97;blue,0}, squiggly, from=7-2, to=8-4]
	\arrow[draw={rgb,255:red,230;green,97;blue,0}, squiggly, from=6-2, to=7-4]
	\arrow[draw={rgb,255:red,93;green,58;blue,155}, squiggly, from=9-6, to=10-4]
	\arrow[draw={rgb,255:red,230;green,97;blue,0}, squiggly, from=7-3, to=8-5]
	\arrow[from=5-2, to=6-2]
	\arrow[from=6-2, to=7-2]
	\arrow[from=7-2, to=8-2]
	\arrow[from=8-2, to=9-2]
	\arrow[from=9-2, to=10-1]
	\arrow[from=10-1, to=11-1]
	\arrow[from=11-1, to=12-1]
	\arrow[from=12-1, to=13-1]
	\arrow[from=13-1, to=14-1]
	\arrow[from=6-3, to=7-3]
	\arrow[from=7-3, to=8-3]
	\arrow[from=7-4, to=8-4]
	\arrow[from=8-3, to=9-3]
	\arrow[from=8-4, to=9-4]
	\arrow[from=8-5, to=9-5]
	\arrow[from=9-5, to=10-4]
	\arrow[from=9-4, to=10-3]
	\arrow[from=9-3, to=10-2]
	\arrow[from=10-2, to=11-2]
	\arrow[from=11-2, to=12-2]
	\arrow[from=12-2, to=13-2]
	\arrow[from=10-3, to=11-3]
	\arrow[from=11-3, to=12-3]
	\arrow[from=10-4, to=11-4]
	\arrow[from=5-2, to=6-3]
	\arrow[from=6-2, to=7-3]
	\arrow[from=7-2, to=8-3]
	\arrow[from=7-3, to=8-4]
	\arrow[from=7-4, to=8-5]
	\arrow[from=8-5, to=9-6]
	\arrow[from=8-4, to=9-5]
	\arrow[from=8-3, to=9-4]
	\arrow[from=8-2, to=9-3]
	\arrow[from=9-2, to=10-2]
	\arrow[from=9-3, to=10-3]
	\arrow[from=9-4, to=10-4]
	\arrow[from=9-5, to=10-5]
	\arrow[from=10-5, to=11-4]
	\arrow[from=10-4, to=11-3]
	\arrow[from=10-3, to=11-2]
	\arrow[from=10-2, to=11-1]
	\arrow[from=11-2, to=12-1]
	\arrow[from=12-2, to=13-1]
	\arrow[from=13-2, to=14-1]
	\arrow[from=11-3, to=12-2]
	\arrow[from=12-3, to=13-2]
	\arrow[from=11-4, to=12-3]
	\arrow[from=1-2, to=2-2]
	\arrow[from=2-2, to=3-2]
	\arrow[from=3-2, to=4-2]
	\arrow[from=4-2, to=5-2]
	\arrow[from=2-2, to=3-3]
	\arrow[from=3-3, to=4-3]
	\arrow[from=3-2, to=4-3]
	\arrow[from=4-2, to=5-3]
	\arrow[from=4-3, to=5-3]
	\arrow[from=4-3, to=5-4]
	\arrow[from=5-3, to=6-3]
	\arrow[from=5-3, to=6-4]
	\arrow[from=5-4, to=6-4]
	\arrow[from=6-4, to=7-4]
	\arrow[from=6-4, to=7-5]
	\arrow[from=7-5, to=8-5]
	\arrow[from=6-3, to=7-4]
	\arrow[from=14-1, to=15-1]
	\arrow[from=15-1, to=16-1]
	\arrow[from=16-1, to=17-1]
	\arrow[from=17-1, to=18-1]
	\arrow[from=18-1, to=19-1]
	\arrow[from=13-2, to=14-2]
	\arrow[from=12-3, to=13-3]
	\arrow[from=13-3, to=14-3]
	\arrow[from=14-3, to=15-3]
	\arrow[from=11-4, to=12-4]
	\arrow[from=12-4, to=13-4]
	\arrow[from=10-5, to=11-5]
	\arrow[from=11-5, to=12-4]
	\arrow[from=12-4, to=13-3]
	\arrow[from=13-4, to=14-3]
	\arrow[from=13-3, to=14-2]
	\arrow[from=14-3, to=15-2]
	\arrow[from=14-2, to=15-1]
	\arrow[from=14-2, to=15-2]
	\arrow[from=15-3, to=16-2]
	\arrow[from=15-2, to=16-2]
	\arrow[from=15-2, to=16-1]
	\arrow[from=16-2, to=17-2]
	\arrow[from=16-2, to=17-1]
	\arrow[from=17-2, to=18-1]
\end{tikzcd}\]}\caption{Closure relations for $q=11$.}\label{figure:Closure11}
\end{figure}

Observe that there are three types of relations shown, as follows:
\begin{itemize}
    \item Bruhat relations are represented by black, straight arrows, and represent relations of the following form: $\gamma_{u,v}\Heleq \gamma_{u+1,v}$ and $\gamma_{u,v}\Heleq \gamma_{u,v+1}$. These are the relations given in \cref{prop:level1Bruhat}.
    \item The orange squiggle arrows indicate closure relations of the form $\gamma_{u+1,v-2}\Heleq \gamma_{u,v}$. These are above the equator, and point down and to the right.
        \item The purple squiggle arrows indicate closure relations of the form
    $ \gamma_{u-2,v+1}\Heleq \gamma_{u,v}$. These are below the equator, and point down and to the left.
\end{itemize}

These last two relations are given in Theorems  \ref{thm:primaryHerelations} and \ref{thm:secondaryHerelations}. 

\subsubsection{Concluding remarks on the closure order}

Based on experiments done in Sage \cite{sagemath}, we believe that the relations in Theorems \ref{thm:primaryHerelations} and \ref{thm:secondaryHerelations} are the only non-Bruhat closure relations between coset representatives whose lengths differ by 1, and that all other closure relations comes from a chain of such relations combined with Bruhat relations. We summarize this in the following conjecture.
\begin{conjecture}\label{conjec:ClosureRelations}
    Let $\gamma_{u_1,v_1}$ and $\gamma_{u_2,v_2}$ be elements of $\mathbf{W}(q-2,2)$ satisfying $\gamma_{u_1,v_1} \Heleq \gamma_{u_2, v_2}$. Then either we have $\gamma_{u_1,v_1} \leq \gamma_{u_2, v_2}$ with respect to the Bruhat order, or $\gamma_{u_1,v_1}$ and $\gamma_{u_2, v_2}$ satisfy the closure relations in \cref{thm:primaryHerelations} or \ref{thm:secondaryHerelations}, or $\gamma_{u_1,v_1}$ and $\gamma_{u_2, v_2}$ are related by a chain of such relations.
\end{conjecture}

We have checked these claims in Sage  for $q\leq 11$, but it was difficult to fully verify these statements for larger values of $q$. Up to $q=20$ we confirmed that we do not get further relations from acting by a single simple transposition $s_k$. Checking the action for all possible $h\in W_{(q-2,2)}$, which has size $2!(q-2)!$, is computationally challenging. For example, even for $q=13$ there are over 79 million 
possible $h$. Nevertheless, our numerical experiments and our extensive work with the action in Equation \eqref{eqn:He_action} lead us to believe that Conjecture \ref{conjec:ClosureRelations} is true and that further work in this area would be fruitful.

\begin{remark}
There are $2\lceil q/2 \rceil - 5$ numbers $j$ that satisfy the conditions of \cref{thm:primaryHerelations}. Moreover, there are $(\lceil q/2 \rceil -3)^2$ pairs $(i,j)$ that satisfy the conditions of \cref{thm:secondaryHerelations}. Combined, this gives $(\lceil q/2 \rceil - 2)^2$ closure relations that are not Bruhat relations. Conjecture~\ref{conjec:ClosureRelations} predicts that this is the total number of closure relations that are not Bruhat relations.
\end{remark}


\section{Product Maps: Relations to Other Unitary Shimura Varieties}\label{sec:prod}

In this section, we begin our study of the interaction between Ekedahl-Oort strata and Newton strata by considering products of abelian varieties. Note that when an abelian variety $A$ decomposes as a product of abelian varieties,  the $p$-torsion group scheme of $A$ also decomposes as a product. Because of this, the natural "product map" from a pair of unitary Shimura varieties to our Shimura variety of interest, $\mathcal{M}(q-2,2)$, will also induce a map in terms of the Ekedahl-Oort strata of these Shimura varieties. In \cref{thm:1x1-complete,thm2x0}, we explicitly describe these induced product maps on the index sets $\textbf{W}(a,b)$ for the Ekedahl-Oort stratifications of the relevant Shimura varieties. Then, we enumerate some Ekedahl-Oort strata that must intersect the supersingular locus of $\mathcal{M}(q-2,2)$, using the product maps and the observation that when two abelian varieties are supersingular, their product is also supersingular.

\subsection{Background}

For any $m_1,m_2,n_1,n_2 \in \mathbx{N}$, there is a natural \emph{product map}
\begin{align*}
  \Phi:  \mathcal{M}(m_1,n_1) \times \mathcal{M}(m_2,n_2) &\to \mathcal{M}(m_1+m_2, n_1 + n_2) \\
    ((A_1, \lambda_1, \iota_1, \xi_1), (A_2, \lambda_2, \iota_2, \xi_2)) &\mapsto (A_1 \times A_2, \lambda_1 \times \lambda_2, \iota_1 \times \iota_2, \xi_1 \times \xi_2),
\end{align*} 
where $\lambda_1 \times \lambda_2$, $\iota_1 \times \iota_2$, and $\xi_1 \times \xi_2$ are (respectively) the natural product polarization, action, and level structure on the abelian variety $A_1 \times A_2$. 

Consider an abelian variety $A$ which decomposes as a product $A \cong A_1 \times A_2$, and note that the $p$-torsion group scheme satisfies $A[p] \cong A_1[p] \times A_2[p]$. Replacing $A_1$ by an abelian variety $B_1$ such that $A_1[p] \cong B_1[p]$  does not affect the Ekedahl-Oort stratum of $A$. 

As a result of this, given Ekedahl-Oort strata $\mathcal{M}(m_1, n_1)_{\omega_1}$ and  $\mathcal{M}(m_2, n_2)_{\omega_2}$, there is a unique stratum $\mathcal{M}(m_1+m_2, n_1+n_2)_{\omega}$ such that:
$$\Phi( \mathcal{M}(m_1, n_1)_{\omega_1} \times \mathcal{M}(m_2, n_2)_{\omega_2} ) \subseteq \mathcal{M}(m_1 + m_2, n_1 + n_2)_{\omega}.$$
In particular, the product map $\Phi$ induces a map $\phi$ on the index sets for the Ekedahl-Oort strata:
$$\phi: \mathbf{W}(m_1,n_1) \times \mathbf{W}(m_2,n_2) \rightarrow \mathbf{W}(m_1+m_2,n_1+n_2),$$
where $\phi(\omega_1,\omega_2)$ is the index for the unique Ekedahl-Oort stratum containing $\Phi(\mathcal{M}(m_1, n_1)_{\omega_1} \times \mathcal{M}(m_2, n_2)_{\omega_2})$.
 
The goal of this section is to explicitly describe the product map on the level of Weyl group cosets under the condition $(m_1+m_2,n_1 + n_2)=(q-2,2)$. This is done by first constructing \emph{standard objects} of the strata $\mathcal{M}(m_1,n_1)_{\omega_1}$ and $\mathcal{M}(m_2,n_2)_{\omega_2}$ and then computing the permutation corresponding to the sum of these standard objects.

Without loss of generality, we assume $n_2\leq n_1$, so that either $n_1=1$ or $n_1=2$.  The former case is treated in \cref{subsec:1x1} and the latter in \cref{subsec:2x0}. 

\subsection{The \texorpdfstring{$1 \times 1$}{} Multiplication Map} \label{subsec:1x1}

\subsubsection{General approach}\label{subsubsec:gen1X1}

In this section we study the product map
$$\Phi:\mathcal{M}(m,1) \times \mathcal{M}(n,1) \rightarrow \mathcal{M}(m+n, 2),$$
under the condition $m+n=q-2$. We compute the induced map on Ekedahl-Oort strata
$$\varphi: \mathbf{W}(m,1) \times \mathbf{W}(n,1) \rightarrow \mathbf{W}(m+n,2).$$
By \cite[Theorem 6.7]{Moonen2001}, the Ekedahl-Oort strata of $\mathcal{M}(m,1)$ are in correspondence with cosets in the quotient $(\mathfrak{S}_{1} \mathfrak{S}_{m}) \setminus \mathfrak{S}_{m+1}.$
Each coset has minimal-length representative given by a cyclic permutation 
$$\gamma_{a+1} := (1,2,\ldots,a+1),$$ for $0 \leq a \leq m$. We therefore have 
$$\mathbf{W}(m,1) = \{ \gamma_{a+1} \; | \; 0 \leq a \leq m\}.$$ 
Note that the length of $\gamma_{a+1}$ is $a$,  and  so the corresponding stratum $\mathcal{M}(m,1)_{\gamma_{a+1}}$ has dimension $a$.

In \cite[Theorem F]{VollaardWedhorn}, the interaction between the Ekedahl-Oort stratification and the Newton stratification of $\mathcal{M}(m,1)$ is completely described. If $a \leq m/2 $, then we have containment: $$\mathcal{M}(m,1)_{\gamma_{a+1}} \subseteq \mathcal{M}(m,1)^{ss}.$$ On the other hand, if $a > m/2$, then we have disjointedness: \[\mathcal{M}(m,1)_{\gamma_{a+1}} \cap \mathcal{M}(m,1)^{ss} = \emptyset.\]
In particular, from here on, given a Dieudonn\'{e} module $M$ arising as the standard object for an  Ekedahl-Oort stratum $\mathcal{M}(m,1)_{\gamma_{a+1}}$, we say that $M$ is \textbf{supersingular} if $a \leq m/2$.

Fix $0 \leq a \leq m$ and $0 \leq b \leq n$. We aim to describe $\varphi(\gamma_{a+1}, \gamma_{b+1})\in \mathbf{W}(m+n, 2)$.
Recall the bijection between $\mathbf{W}(a,b)$ and Ekedahl-Oort strata of $\mathcal{M}(a,b)$ outlined below \cref{thm:Moonenbij}. Let the Dieudonn\'{e} module $M$ be the standard object in the Ekedahl-Oort stratum $\mathcal{M}(m,1)_{\gamma_{a+1}}$. Recall that the $\mathbx{F}_{p^2}$-action induces a decomposition $M= M_1 \oplus M_2$. We fix a final filtration $W_\bullet$ of $M$, which induces filtrations $C_{i,\bullet} = W_\bullet \cap M_i$ of $M_i$ for $i=1,2$.
In our specific case, we have $(a,b)=(m,1)$ and $\dim(M_1[F])=1$, so the function 
$$\eta^M_1(j) = \dim ( C_{1,j} \cap M[F])$$ 
jumps, i.e., increases by 1, at exactly one index. This index is $j_1=a+1$ by the bijection in \cref{thm:Moonenbij}.

Similarly, we let $\gamma_{b+1} \in \mathbf{W}(n,1)$ represent an Ekedahl-Oort stratum in $\mathcal{M}(n,1)$, with standard object $N$. Let $D_{i,\bullet}$ be filtrations of $N_i$ for $i=1,2$ coming from a final filtration of $N$. Let $L:=M \oplus N$ and let $E_{i,\bullet }$ be the filtrations of $L_i$ coming from a final filtration of $L$. Our goal is to determine the Weyl group coset corresponding to the Dieudonn\'{e} module $L$ via \cref{thm:Moonenbij}.  We do this by computing the function
$$\eta^L_1(j)= \dim (E_{1,j} \cap L[F] ).$$  
As noted in \cref{sec:WeylGroupCosets}, it is sufficient to compute $\eta_1$ in our work, and so this determines the permutation $\gamma_{u,v} = \phi(\gamma_{a+1},\gamma_{b+1})$ representing the Ekedahl-Oort stratum of $L$. The signature of $L$ is $(m+n,2)$, and so we have
\[\eta^L_1(q) = \dim(L_1[F]) = 2.\]
Since the function $\eta^L_1$ is clearly non-decreasing, it suffices to find the two integers where $\eta^L_1$ jumps. Then this function $\eta_1^L$ corresponds to a Weyl group coset $\gamma_{u,v}$. It can be seen from the description of the standard object in \cref{lem:SOuv} that $u$ and $v$ are precisely the two integers where $\eta_1^L$ jumps.

Depending on whether $M$ and $N$ are supersingular or not, there are slight differences in our method of determining the two places where $\eta^L_1$ jumps. These cases are treated separately in \cref{prop:SSxSScomplete,prop:NSxNScomplete,prop:NSxSScomplete} (and summarized in \cref{thm:1x1-complete}). In \cref{sec:1x1canonicalfiltration}, the necessary parts of the canonical filtration of $L$ are constructed. In \cref{sec:1x1EOstrata}, this information is used to determine the jumps of $\eta^L_1$, which in turn yields the resulting Weyl group coset representative $\gamma_{u,v}$ of $L$, such that
$$ \Phi \left( \mathcal{M}(m,1)_{\gamma_{a+1}} \times \mathcal{M}(n,1)_{\gamma_{b+1}} \right) \subseteq \mathcal{M}(m+n,2)_{\gamma_{u,v}}.$$ 

\subsubsection{The canonical filtration} \label{sec:1x1canonicalfiltration}

Using the standard objects of $M$ and $N$, we form the canonical filtration of $L$, which is crucial in computing the Ekedahl-Oort stratum. Let $M$ be the standard object of the Ekedahl-Oort stratum $\mathcal{M}(m,1)_{\gamma_{a+1}}$. \cite[4.9]{Moonen2001} provides the action of $F$ and $V$ on a basis $\{ e_{i,j} \; | \; 1\leq i \leq 2 \; , \; 1 \leq j \leq q \}$ of $M$. 
\begin{align*}
&F(e_{1,j}) = \begin{cases} e_{2,j} &\hbox{if $1 \leq j \leq a$,} \\ 0 &\hbox{if $j=a+1$,}  \\  e_{2,j-1} &\hbox{if $j>a+1$,} \end{cases}
&V(e_{1,j}) &= \begin{cases} 0 &\hbox{if $j=1$,}  \\ e_{2,j-1} &\hbox{if $1 < j \leq m+1-a$,} \\ e_{2,j} &\hbox{if $j> m+1-a $,} \end{cases} \\
&F(e_{2,j}) = \begin{cases} 0 &\hbox{if $j\neq m+1-a$,}  \\ e_{1,1} &\hbox{if $j=m+1-a$,}\end{cases}
&V(e_{2,j}) &= \begin{cases} 0 &\hbox{if $j\neq m+1$,}  \\ e_{1,a+1} &\hbox{if $j=m+1$}. \end{cases}
\end{align*}

The following lemma describes the action of $F$ and $V^{-1}$ on the subspaces $\langle e_{i,l} \; | \; 1 \leq l \leq j \rangle$. Since applying $F$ or $V^{-1}$ always gives another such subspace, we pick a final filtration $W_\bullet$ of $M$ such that the filtrations $C_{i,\bullet} = W_\bullet \cap M_i$ are given by
$$ C_{i,j} = \langle e_{i,l} \; | \; 1 \leq l \leq j \rangle.$$ 

\begin{lemma} \label{lem:SOM}
$F$ and $V^{-1}$ have the following actions on the spaces $C_{i,j}$:
\begin{align*}
&F(C_{1,j}) = \begin{cases} C_{2,j} &\hbox{if $j \leq a$,} \\ C_{2,j-1} &\hbox{if $j>a$,}  \end{cases}
&V^{-1}(C_{1,j}) \cap M_2 &= \begin{cases} C_{2,m-1} &\hbox{if $j\leq a$,}  \\ C_{2,m} &\hbox{if $j>a$,} \end{cases} \\
&F(C_{2,j}) = \begin{cases} C_{1,0} &\hbox{if $j < m+1-a$,}  \\ C_{1,1} &\hbox{if $j \geq m+1-a$,} \end{cases}
&V^{-1}(C_{2,j}) \cap M_1 &= \begin{cases} C_{1,j+1} &\hbox{if $j < m+1-a$,}  \\ C_{1,j} &\hbox{if $j \geq  m+1-a$}. \end{cases}
\end{align*}
\end{lemma}
\begin{proof}
This is a direct application of the description, given above, of the standard object $M$.
\end{proof}

The same result applies to the standard object $N$ of the Ekedahl-Oort stratum $\mathcal{M}(n,1)_{\gamma_{b+1}}$, upon replacing $a$ by $b$ and $m$ by $n$. 

For the remainder of this section, let $s_1:=\min\{a+1,m+1-a\}$ and $s_2:=\max\{a,m+1-a\}$.

\begin{lemma} \label{lemVF}
For $1\leq j\leq q$,
\[V^{-1}(F(C_{1,j})) \cap M_1 = 
\begin{cases}
    C_{1,j+1} & j < s_1,\\[0.5em]
    C_{1,j-1} & j > s_2,\\[0.5em]
    C_{1,j} & j = s_1, s_2.    
\end{cases}\]

\end{lemma}
\begin{proof}
Suppose $j < s_1 = \min\{a+1, m+1-a\}$. Using \cref{lem:SOM}, we obtain $F(C_{1,j})= C_{2,j}$ and $V^{-1}(C_{2,j}) \cap M_1 = C_{1,j+1}$ as desired.

We check the second assertion for two cases of $s_1$. In the case $s_1=a+1 \leq m+1-a$ we get $F(C_{1,a+1})= C_{2,a}$ and $V^{-1}(C_{2,a}) \cap M_1 = C_{1,a+1}$. In the case $s_1=m+1-a < a+1$, we get $F(C_{1,m+1-a})=C_{2,m+1-a}$ and $V^{-1}(C_{2,m+1-a}) \cap M_1 = C_{1,m+1-a}$. 

We now treat the analogue with decreasing index: assume $j > s_2=\max\{a,m+1-a\}$. In this case, we have $F(C_{1,j}) = C_{2,j-1}$ and, subsequently, $V^{-1}(C_{2,j-1}) \cap M_1 =C_{1,j-1}$.

Finally, we prove the last assertion for both values of $s_2$. In the case $s_2=a \geq m+1-a$, we get $F(C_{1,a})=C_{2,a}$ and $V^{-1}(C_{2,a}) \cap M_1 = C_{1,a}$. In the case $s_2=m+1-a > a$ we obtain $F(C_{1,m+1-a})=C_{2,m-a}$ and $V^{-1}(C_{2,m-a}) \cap M_1=C_{1,m+1-a}$, as desired.
\end{proof}

We now prove a corollary that allows us to `move up' from $C_{1,0}$ to $C_{1,s_1}$, step by step.

\begin{corollary} \label{cormovingup}
For $c \in \mathbx{Z}_{\geq 0}$, we have $$(V^{-1}F)^c(0) \cap M_1 = C_{1,\min\{c,s_1\}}.$$
\end{corollary}
\begin{proof} 
    Apply \cref{lemVF} for $0 \leq j \leq s_1$ repeatedly, so that the result follows by induction starting at $j=0$. 
\end{proof}

Similarly we can `move down' from $C_{1,m+1}$ to $C_{1,s_2}$, as the following corollary records. 

\begin{corollary} \label{cormovingdown}
For $c \in \mathbx{Z}_{\geq 0}$, we have $$(V^{-1}F)^c(M) \cap M_1 = C_{1,\max\{m+1-c,s_2\}}.$$
\end{corollary}
\begin{proof} 
    Apply \cref{lemVF} for $s_2 \leq l \leq m+1$ repeatedly, starting at $l=m+1$ and proceeding by decreasing induction. 
\end{proof}

In the rest of this section, we use these corollaries  to construct the crucial parts of the filtration $E_{1,\bullet}$ of $L_1 = M_1 \oplus N_1$, induced by intersecting a final filtration of $L$ with $L_1$. Let $s_1$ and $s_2$ be defined as above and let $u_1:=\min\{b+1,n+1-b\}$ and $u_2:=\max\{b,n+1-b\}$. 

\begin{lemma} \label{lemmovingupdown}
For $c \in \mathbx{Z}_{\geq 0}$, we have
\begin{align*}
(V^{-1}F)^c(0) \cap L_1 &= C_{1,\min\{c,s_1\}} \oplus D_{1,\min \{c,u_1\}} = E_{1,\min\{c,s_1\} + \min\{c,u_1\}}, \text{ and } \\
(V^{-1}F)^c(L) \cap L_1 &= C_{1,\max\{m+1-c,s_2\}} \oplus D_{1,\max\{n+1-c,u_2\}} = E_{1,\max\{m+1-c,s_2\} + \max\{n+1-c,u_2\}}. 
\end{align*}
\end{lemma}
\begin{proof}
This is a straightforward application of \cref{cormovingup} and \cref{cormovingdown}. For the first equality, note
$$ (V^{-1}F)^c(0) \cap L_1 = \left((V^{-1}F)^c(0) \cap M_1 \right) \oplus \left( (V^{-1}F)^c(0) \cap N_1\right).$$
Both summands are computed in \cref{cormovingup}. Similarly, for the second equality observe
$$ (V^{-1}F)^c(L) \cap L_1 = \left((V^{-1}F)^c(M) \cap M_1 \right) \oplus \left( (V^{-1}F)^c(N) \cap N_1\right).$$
Both summands are computed in \cref{cormovingdown}. 
\end{proof}

In some cases, \cref{lemmovingupdown} determines every value of $\eta^L_1(j) = \dim( E_{1,j} \cap L[F])$, providing all the information we need. In general, we need to construct more subspaces of $L_1$. For that, we turn to the following lemma.

\begin{lemma} \label{lem:F^2}
Suppose one of the following conditions holds:
\begin{enumerate}[label= (\roman*)]
\item $M$ and $N$ are both supersingular and $a > b$; \label{case:a>b}
\item $M$ and $N$ are both not supersingular and $m-a<n-b$;\, or \label{case:m-a<m-b}
\item $M$ is not supersingular and $N$ is supersingular. \label{case:NSxSS}
\end{enumerate}
Then we have that
\[
    E_{1,1} = C_{1,1} \oplus D_{1,0} \;\;\;\text{and} \;\;\;
    E_{1,q-1} =C_{1,m} \oplus D_{1,n+1}.
\]
\end{lemma}
\begin{proof}
We claim that, in all three cases, it suffices to construct a subspace of the form $C_{2,l} \oplus D_{2,n-b}$, for some $l> m-a$, in the filtration $E_{2,\bullet}$. We then show that such a subspace can be constructed in all three cases. 

Assume that $C_{2,l} \oplus D_{2,n-b}$ has been constructed with $l >m-a$. Then \cref{lem:SOM} yields
$$F\left( C_{2,l} \oplus D_{2,n-b} \right) = C_{1,1} \oplus D_{1,0} = E_{1,1},$$
proving the claim. For the second statement, we first construct $E_{2,q-1}$ using symplectic complements. Note that $M$ and $N$ are each equipped with a symplectic pairing coming from the polarization of the abelian varieties, and so $L$ is equipped with the product of these two pairings. By the Rosati involution condition, we have $M_1^\perp = M_1$. Furthermore, we may assume that the final filtration $W_\bullet$ of $M$ is stable under taking symplectic complements. Therefore, for any $1\leq j \leq q$ we have $C_{1,j}^\perp = M_1 \oplus C_{2,q-j}$ and likewise for $N$. This implies
\[
E^\perp_{1,1} \cap L_2 = (C_{1,1}^\perp \cap M_2) \oplus (D_{1,0}^\perp \cap N_2) 
= C_{2,m} \oplus D_{2,n+1} = E_{2,q-1}.
\]
All three cases imply $a>0$, and so applying \cref{lem:SOM} yields
$$V^{-1}(E_{2,q-1}) = C_{1,m} \oplus D_{1,n+1} = E_{1,q-1},$$
which finishes the proof of the second statement. 

We now construct $C_{2,l} \oplus D_{2,n-b}$ in each individual case. In case \ref{case:a>b}, \cref{lemmovingupdown} allows us to move down to 
$$(V^{-1}F)^b(L) \cap L_1 = C_{1,m+1-b} \oplus D_{1,n+1-b} = E_{1,q-2b}.$$
Here we have used $m+1-b > m+1-a$, since $a>b$. Furthermore, $m+1-a \geq a$ since $M$ is supersingular, and $n+1-b \geq b$ since $N$ is supersingular. Together these observations imply $ m+1-b \geq s_2$ and $ n+1-b \geq u_2$. We then construct
$$ F(C_{1,m+1-b} \oplus D_{1,n+1-b}) = C_{2,m-b} \oplus D_{2,n-b}.$$  

Since $m-b>m-a$, we are done. 

In case \ref{case:m-a<m-b}, \cref{lemmovingupdown} allows us to move up to 
$$(V^{-1}F)^{n-b}(0) \cap L_1 = C_{1,m+1-a} \oplus D_{1,n-b} = E_{1,m+n+1 -a-b}.$$
where we use: $s_1=m+1-a$ and $m-a<n-b$. Applying $F$, using $m+1-a \leq a$ and $n-b \leq b$, yields
$$F\left(C_{1,m+1-a} \oplus D_{1,n-b}\right) = C_{2,m+1-a} \oplus D_{2,n-b}.$$
This subspace satisfies our requirements, as $m+1-a>m-a$. 

Finally, in case \ref{case:NSxSS}, \cref{lemmovingupdown} allows us to move down to 
$$ C_{1,s_2} \oplus D_{1,u_2} = C_{1,a} \oplus D_{1,n+1-b}.$$ Applying $F$ yields
$$F(C_{1,a} \oplus D_{1,n+1-b}) = C_{2,a} \oplus D_{2,n-b},$$
which satisfies our requirement since $M$ is not supersingular, implying $a>m-a$.
\end{proof}

After constructing $E_{1,1}$ and $E_{1,q-1}$ if necessary, under the conditions of \cref{lem:F^2}, we use the proof of \cref{lemmovingupdown} to construct 
\begin{align}
(V^{-1}F)^c(E_{1,1}) \cap L_1 &= C_{1,\min\{1+c,s_1\}} \oplus D_{1,\min \{c,u_1\}} = E_{1,\min\{1+c,s_1\} + \min\{c,u_1\}}, \text{ and } \label{eq:movingupadapted} \\
(V^{-1}F)^c(E_{1,q-1}) \cap L_1 &= C_{1,\max\{m-c,s_2\}} \oplus D_{1,\max\{n+1-c,u_2\}} = E_{1,\max\{m-c,s_2\} + \max\{n+1-c,u_2\}}. \label{eq:movingdownadapted}
\end{align}

The information that these lemmata provide about the canonical filtration of $L$ is sufficient to compute its Ekedahl-Oort stratum.

\subsubsection{Ekedahl-Oort strata} \label{sec:1x1EOstrata}

We now have the tools to completely describe the product map on the level of Weyl group cosets
$$\phi: \mathbf{W}(m,1) \times \mathbf{W}(n,1) \rightarrow \mathbf{W}(m+n,2).$$
In particular, we compute $\phi(\gamma_{a+1}, \gamma_{b+1})$, where $\mathcal{M}(m,1)_{\gamma_{a+1}}$ is the unique Ekedahl-Oort stratum of dimension $a$ and likewise for $\gamma_{b+1}$. Let $M$ be a standard object for the stratum $\gamma_{a+1}$ and let $N$ be a standard object for the stratum $\gamma_{b+1}$. All possibilities for $M$ and $N$ (up to exchanging $M$ and $N$) are treated in this section. Recall that $M$ is supersingular if and only if $a\leq m/2$ or, equivalently, $s_1=a+1$. $N$ is supersingular if and only if $b \leq n/2$ or, equivalently, $u_1=b+1$. 

We first treat the case when $M$ and $N$ are both supersingular in full detail. Exchanging $M$ and $N$ if necessary, we assume $a \geq b$. 

\begin{proposition} \label{prop:SSxSScomplete}
Assume $M$ and $N$ are both supersingular, with $a \geq b$. Then the Ekedahl-Oort stratum of $L$ is represented by the Weyl group coset of $\gamma_{2b+1,2b+2}$ if $a=b$ and $\gamma_{2b+2,a+b+2}$ if $a>b$. 
\end{proposition}
\begin{proof}
We compute the two places $j_1$ and $j_2$ where $\eta^L_1$ jumps.
Using \cref{lemmovingupdown} we `move up' to construct the following parts of the canonical filtration:
\begin{align*}
    (V^{-1}F)^b(0) \cap L_1 &= C_{1,b} \oplus D_{1,b} = E_{1,2b},  \\
    (V^{-1}F)^{b+1}(0) \cap L_1 &= C_{1,b+1} \oplus D_{1,b+1} = E_{1,2b+2},  \\
    (V^{-1}F)^a(0) \cap L_1 &= C_{1,a} \oplus D_{1,b+1} = E_{1,a+b+1} \hspace{1cm} \hbox{in the case $a>b$,} \\
    (V^{-1}F)^{a+1}(0) \cap L_1 &= C_{1,a+1} \oplus D_{1,b+1} = E_{1,a+b+2}.
\end{align*}

We use the explicit description $M_1[F]=\langle e_{1,a+1} \rangle$ (see \cref{lem:SOM}), and similarly for $N_1[F]$. Thus, we compute the following values of $\eta^L_1$:
\begin{align*}
\eta^L_1(2b) &= 0, \\
\eta^L_1(2b+2) &= \begin{cases} 1 &\hbox{if $a>b$,} \\ 2 &\hbox{if $a=b$,} \end{cases} \\
\eta^L_1(a+b+1) &= 1, \\
\eta^L_1(a+b+2) &= 2.
\end{align*}

Hence $\eta^L_1$ jumps at $j_2=a+b+2$. In the case $a=b$, it follows that $\eta^L_1$ must also jump at $j_1=2b+1$, since $\eta^L_1(2b)=0$. 

When $a>b$, it is not yet determined whether the first jump occurs at $2b+1$ or $2b+2$ and more work is required.  In this case, \cref{lem:F^2} and Equation~\eqref{eq:movingupadapted} imply $E_{1,1} = C_{1,1} \oplus D_{1,0},$ and so
\begin{align*}
(V^{-1}F)^b(E_{1,1}) &= C_{1,\min\{b+1,a+1,m+1-a\}} \oplus D_{1,\min\{b,a+1,m+1-a\}} \\
&= C_{1,b+1} \oplus D_{1,b} = E_{1,2b+1}.
\end{align*}
Using the explicit description of $L_1[F]$, we infer $\eta^L_1(2b+1)=0$ and hence $\eta^L_1$ jumps at $j_1=2b+2$. 

Finally, we prove the statement about the Weyl group coset representative $\gamma_{u,v}$. Recall that $\omega=\gamma_{u,v}$ is determined by the jumps of $\eta_L^1$ via Equation~\eqref{eqn:omegafromDD}. Comparing this with the definition of $\gamma_{u,v}$ in Corollary \ref{prop:EO-comma-2} yields $u=j_1$ and $v=j_2$, completing the proof.
\end{proof}

We now treat the case when $M$ and $N$ are both not supersingular, which is essentially the mirror image of the preceding case. Exchanging $M$ and $N$ if necessary, assume that $m-a \leq n-b$.

\begin{proposition} \label{prop:NSxNScomplete}
Assume $M$ and $N$ are both not supersingular, with $m-a \leq n-b$. Then the Ekedahl-Oort stratum of $L$ is represented by the Weyl group coset of $\gamma_{a+b+1,2a+n-m+2}$.
\end{proposition}
\begin{proof}

Analogous to Proposition \ref{prop:SSxSScomplete}, we use \cref{lemmovingupdown} to compute $(V^{-1}F)^{t}(L) \cap L_1$ for relevant values of $t$ and obtain the following values for $\eta^L_1$:
\begin{align*}
\eta^L_1(a+b) &= 0, \\
\eta^L_1(a+b+1) &= 1, \\
\eta^L_1(2a+n-m) &= 1, \text{ and}\\
\eta^L_1(2a+n-m+2) &= 2.
\end{align*}
This implies that $\eta^L_1$ jumps at $j_1=a+b+1$. In the case $m-a=n-b$, the second jump is at $j_2=2a+n-m+2=a+b+2.$ 

When $m-a<n-b$, \cref{lem:F^2} and Equation~\eqref{eq:movingdownadapted} imply $E_{1,q-1} =C_{1,m} \oplus D_{1,n+1}$, and so
\begin{align*}
    (V^{-1}F)^{m-a}(E_{1,q-1}) &= C_{1,a} \oplus D_{1,n+1-m+a} = E_{1,2a+n-m+1}.
\end{align*}

Using the explicit description of $L_1[F]$, it follows that $\eta^L_1(2a+n-m+1)=1$. Hence in that case we also have $j_2=2a+n-m+2$.

The statement about Weyl group coset representatives follows from the formulae $u=j_1$ and $v=j_2$.
\end{proof}

Finally, we treat the case when exactly one of the summands $M$ and $N$ is supersingular. Exchanging $M$ and $N$ if necessary, we assume that $M$ is not supersingular and $N$ is supersingular.

\begin{proposition} \label{prop:NSxSScomplete}
Assume $M$ is not supersingular and $N$ is supersingular. Then the Ekedahl-Oort stratum of $L$ is represented by the Weyl group coset of $\gamma_{b+2 + \min\{b,m-a\},a+2+\max\{n-m+a,n-b\}}$.
\end{proposition}
\begin{proof}
Analogous to Proposition \ref{prop:SSxSScomplete}, we use \cref{lemmovingupdown}, \cref{lem:F^2}, Equation~\eqref{eq:movingupadapted} and Equation~\eqref{eq:movingdownadapted} to compute powers of $V^{-1}F$ on relevant subspaces in order to understand the map $\eta^L_1$.

We determine that the jumps of $\eta^L_1$ are
\begin{align*}
    u&= j_1 = b+2+\min\{b,m-a\}, \text{ and} \\
    v&= j_2 = a+2+\max\{n-m+a,n-b\}. \qedhere
\end{align*}
\end{proof}

\cref{prop:SSxSScomplete}, \cref{prop:NSxNScomplete} and \cref{prop:NSxSScomplete} are summarized in \cref{thm:1x1-complete}.

\begin{theorem}\label{thm:1x1-complete}
The Ekedahl-Oort stratum of $L=M \oplus N$ is given as follows.
 \begin{enumerate}[label=(\Alph*)]
    \item\label{thm:productmap-A} If $M$ and $N$ are both supersingular, with $a \geq b$, then the Ekedahl-Oort stratum of $L$ is represented by the Weyl group coset of $\gamma_{2b+1,2b+2}$ if $a=b$ and $\gamma_{2b+2,a+b+2}$ if $a>b$.
    \item\label{thm:productmap-B} If $M$ and $N$ are both not supersingular, with $m-a \leq n-b$, then the Ekedahl-Oort stratum of $L$ is represented by the Weyl group coset of $\gamma_{a+b+1,2a+n-m+2}$. 
    \item\label{thm:productmap-C} If $M$ is not supersingular and $N$ is supersingular. Then the Ekedahl-Oort stratum of $L$ is represented by the Weyl group coset of $\gamma_{b+2 + \min\{b,m-a\},a+2+\max\{n-m+a,n-b\}}$.
\end{enumerate}

\end{theorem}

\cref{thm:1x1-complete} has the following implication on the supersingular locus $\mathcal{M}(q-2,2)^{ss}$.

\begin{corollary} \label{cor:1x1sslocus}
Let $m,n,a,b$ be integers satisfying:
\[m + n = q - 2,\quad a \leq m/2,\quad b \leq n/2,\quad \text{ and } \quad a \geq b.\]
Then the Ekedahl-Oort stratum
$$\gamma_{u,v} = \begin{cases} \gamma_{2b+1,2b+2} &\hbox{if $a=b$,} \\  \gamma_{2b+2,a+b+2} &\hbox{if $a>b$} \end{cases} $$
intersects $\mathcal{M}(q-2,2)^{ss}$.
\end{corollary}
\begin{proof}
A point $(A,\lambda,\iota, \xi)$ in the intersection $\mathcal{M}(q-2,2)_{\gamma_{u,v}} \cap \mathcal{M}(q-2,2)^{ss}$ is constructed as the product
$$(A,\lambda,\iota,\xi) = (A_1 \times A_2, \lambda_1 \times \lambda_2, \iota_1 \times \iota_2, \xi_1 \times \xi_2),$$
where$
    (A_1,\lambda_1,\iota_1,\xi_1) \in \mathcal{M}(m,1)_{\gamma_{a+1}} \text{ and }
    (A_2,\lambda_2,\iota_2,\xi_2) \in \mathcal{M}(n,1)_{\gamma_{b+1}}.$
It follows from \cref{thm:1x1-complete}~\ref{thm:productmap-A} that $(A,\lambda,\iota,\xi)$ lies in the Ekedahl-Oort stratum $\mathcal{M}(q-2,2)_{\gamma_{u,v}}$. Moreover, $A=A_1 \times A_2$ is supersingular, as $A_1$ and $A_2$ are supersingular by the conditions $a \leq m/2$ and $b \leq n/2$ respectively.
\end{proof}

\subsection{The \texorpdfstring{$2\times 0$}{} Multiplication Map} \label{subsec:2x0}

\subsubsection{General approach}\label{subsubsec:gen2x0}

We now study the product map 
$$\Phi: \mathcal{M}(m,2) \times \mathcal{M}(n,0) \to \mathcal{M}(q-2,2),$$
where $m+n=q-2$. Let $\gamma_{u,v}$ represent an Ekedahl-Oort stratum of $\mathcal{M}(m,2)$, with standard object $M$. It is well known that the Shimura variety $\mathcal{M}(n,0)$ consists only of the superspecial Ekedahl-Oort stratum $\mathcal{M}(n,0)_{\mathrm{id}}= \mathcal{M}(n,0)$ which is characterized by the property $F^2=0$. The standard object of that stratum is $N^n= \bigoplus_{l=1}^n N$, where $N$ is the mod-$p$ Dieudonn\'{e} module arising from a supersingular elliptic curve that has an action of $\mathcal{O}_K$ with signature $(1,0)$. Hence, $\mathcal{M}(n,0) = \mathcal{M}(n,0)^{ss}$. Through a study of the sum $L:= M \oplus N^n$, we compute the permutation $\gamma_{y,s} := \phi(\gamma_{u,v} , \mathrm{id})$, such that
$$ \Phi(\mathcal{M}(m,2)_{\gamma_{u,v}} \; , \mathcal{M}(n,0)_{\mathrm{id}}) \subseteq \mathcal{M}(q-2,2)_{\gamma_{y,s}}. $$
As a result, we obtain information about whether $\mathcal{M}(q-2,2)_{\gamma_{y,s}}$ intersects the supersingular locus. If $\mathcal{M}(m,2)_{\gamma_{u,v}}$ intersects the supersingular locus, then so does $\mathcal{M}(q-2,2)_{\gamma_{y,s}}$. Conversely, if $\mathcal{M}(m,2)_{\gamma_{u,v}}$ is not contained in the supersingular locus, then neither is $\mathcal{M}(q-2,2)_{\gamma_{u,v}}$. While we do not always know how $\mathcal{M}(m,2)_{\gamma_{u,v}}$ interacts with the supersingular locus, our methods allow us to draw conclusions in some cases. 

Our method for computing $\phi(\gamma_{u,v}, \mathrm{id})$ is essentially the same as the method used in \cref{subsec:1x1}. We describe the standard object $M$ corresponding to $\mathcal{M}(m,2)_{\gamma_{u,v}}$ and then compute the permutation corresponding to $L=M \oplus N^n$. In this context, set $q_M:=m+2$ to be the dimension of $M$. As before, let $C_{i,\bullet}$ be filtrations of $M_i$ coming from a final filtration of $M$. In the same way, define the filtrations $D_{i,\bullet}$ of $N_i$ and $E_{i,\bullet}$ of $L_i$. 

As before, define $\eta^L_1(j):=\dim (E_{1,j} \cap L[F])$, and in this setting $L_1[F] = M_1[F]=\langle e_{1,u},e_{1,v} \rangle$.
It suffices to pinpoint the two places where $\eta^L_1$ increases, as this determines the permutation $\gamma_{y,s}$ representing the Ekedahl-Oort stratum of $L$ by Equation \eqref{eqn:omegafromDD}.

\subsubsection{The canonical filtration}

We construct parts of the canonical filtration of $L$ by analysing the action of $F$ and $V^{-1}$ on both $M$ and $N$. 
First, $F$ and $V$ act on $N$ as follows:
\begin{equation} \label{eqssKraft}
    F(N_1)=V(N_1)=N_2 \text{ and } F(N_2)=V(N_2)=0.
\end{equation}

The following lemma gives an analogous description for the action of $F$ and $V^{-1}$ on $M$.

\begin{lemma} \label{lem:SOM2x0}
    $F$ and $V^{-1}$ have the following actions on the spaces $C_{i,j}$:
\begin{align*}
&F(C_{1,j}) = \begin{cases} C_{2,j} &\hbox{if $j < u$,} \\ C_{2,j-1} &\hbox{if $u\leq j < v$,} \\ C_{2,j-2} &\hbox{if $j \geq v$,}  \end{cases}
&V^{-1}(C_{1,j}) \cap M_2 &= \begin{cases} C_{2,q_M-2} &\hbox{if $j < u$,}  \\ C_{2,q_M-1} &\hbox{if $u \leq j < v$,} \\ C_{2,q_M} &\hbox{if $j \geq v$,} \end{cases} \\
&F(C_{2,j}) = \begin{cases} C_{1,0} &\hbox{if $j \leq q_M-v$,}  \\ C_{1,1} &\hbox{if $q_M-v < j \leq q_M-u$,} \\ C_{1,2} &\hbox{if $j > q_M-u$,} \end{cases} 
&V^{-1}(C_{2,j}) \cap M_1 &= \begin{cases} C_{1,j+2} &\hbox{if $j \leq q_M-v$,}  \\ C_{1,j+1} &\hbox{if $q_M-v < j \leq q_M-u$,} \\ C_{1,j} &\hbox{if $j > q_M-u$}. \end{cases}
\end{align*}
\end{lemma}

\begin{proof}
The result follows from applying \cref{lem:SOuv}.
\end{proof}

A variant of `moving up' can be used to construct parts of the canonical filtration of $L$. To record this, define the multiset $S_{\textup{up}}=\{u, v, q_M-v+2, q_M-u+1\}$ and let $r_2:= \min (S_{\textup{up}}\setminus \set{\min S_{\textup{up}}})$, i.e., the second element from the multiset when the elements are listed in increasing order. 

\begin{lemma} \label{lem:2x0up}
For a sufficiently large integer $c$, we have that
$(V^{-1}F)^c (0) \cap L_1= C_{1,r_2} \oplus D_{1,0}^n.$
\end{lemma}
\begin{proof}
We compute both summands of 
\begin{equation} \label{eq:2x0upsummands}
(V^{-1}F)^c (0) \cap L_1 = ((V^{-1}F)^c (C_{1,0}) \cap M_1) \oplus ((V^{-1}F)^c(D_{1,0}) \cap N_1)^n.
\end{equation}

First note that $F(D_{1,0})= D_{2,0}$ and $V^{-1}(D_{2,0}) \cap N_1=D_{1,0}$ by Equation~\eqref{eqssKraft}. Hence $$V^{-1}(F(D_{1,0})) \cap N_1 =D_{1,0}$$ which implies the second summand of Equation~\eqref{eq:2x0upsummands}. We use \cref{lem:SOM2x0} to compute the first summand.

First, we assume $j<r_2$ and show
\begin{equation} \label{eq:2x0upmoving}
V^{-1}(F(C_{1,j})) \cap M_1 \supseteq C_{1,j+1}.
\end{equation}
We begin by computing
$$ F(C_{1,j}) = \begin{cases} 
C_{2,j} &\hbox{if $j < u$,} \\
C_{2,j-1} &\hbox{if $u \leq j < v$}.
\end{cases}$$
In the case $j < u$, one verifies
$$V^{-1}(F(C_{1,j})) \cap M_1= V^{-1}(C_{2,j}) \cap M_1 = \begin{cases}
C_{1,j+2} & \hbox{if $j \leq q_M-v$}, \\
C_{1,j+1} & \hbox{if $j > q_M-v$},
\end{cases}$$
so that Equation~\eqref{eq:2x0upmoving} is satisfied. In the case $j\geq u$, it follows from the definition of $r_2$ that $j< q_M-v+2$, and therefore 
$$V^{-1}(F(C_{1,j})) \cap M_1 = V^{-1}(C_{2,j-1}) \cap M_1= C_{1,j+1},$$
by \cref{lem:SOM2x0} and, again, Equation~\eqref{eq:2x0upmoving} is satisfied.
Thus applying $V^{-1}F$ increases the index in $C_{1,\bullet}$ for $j<r_2$. To prove it does not increase further, we show
\begin{equation} \label{eq:2x0upnotmoving}
V^{-1}(F(C_{1,r_2})) \cap M_1 = C_{1,r_2}.
\end{equation}
First, \cref{lem:SOM2x0} yields
$$ F(C_{1,r_2}) = \begin{cases} 
C_{2,r_2} &\hbox{if $r_2 <u$}, \\
C_{2,r_2-1} &\hbox{if $u \leq r_2  < v$}, \\
C_{2,r_2-2} &\hbox{if $r_2=v$}.
\end{cases}$$
In the first case we have $r_2 < u < v$ and it follows that $r_2 = q_M-u+1$. Hence
$$V^{-1}(F(C_{1,r_2})) \cap M_1 = V^{-1}(C_{2,r_2}) \cap M_1 = C_{1,r_2}.$$
In the second case we have $u \leq r_2 < v$ and it follows that $q_M-v+2 \leq r_2 \leq q_M-u+1$. Therefore 
$$V^{-1}(F(C_{1,r_2})) \cap M_1 = V^{-1}(C_{2,r_2-1}) \cap M_1 = C_{1,r_2}.$$
Finally, in the third case we have $r_2=v$ and it follows that $r_2 \leq q_M-v+2$. We obtain
$$V^{-1}(F(C_{1,r_2})) \cap M_1 = V^{-1}(C_{2,r_2-2}) \cap M_1 = C_{1,r_2}.$$

We conclude that applying $V^{-1}F$ allows us to 'move up' along $C_{1,\bullet}$, precisely until we reach $C_{1,r_2}$. This proves the lemma.
\end{proof}

Analogously, we record the effect of `moving down'. Define the multiset $S_{\text{down}}:= \{u-1,v-1,q_M-v+2,q_M-u+1 \}$ and let $r_3 : = \max (S_{\textup{down}}\setminus \set{\max S_{\textup{down}}})$, i.e., the second element from the multiset when the elements are listed in decreasing order. 

\begin{lemma} \label{lem:2x0down}
For a sufficiently large integer $c$, we have
$(V^{-1}F)^c (L) \cap L_1= C_{1,r_3} \oplus D_{1,1}^n$.
\end{lemma}
\begin{proof}
The proof is similar to the proof of \cref{lem:2x0up}.  We compute both summands of 
\begin{equation} \label{eq:2x0downsummands}
(V^{-1}F)^c (L) \cap L_1 = ((V^{-1}F)^c (C_{1,q_M}) \cap M_1) \oplus ((V^{-1}F)^c(D_{1,1}) \cap N_1)^n.
\end{equation}

First, using Equation~\eqref{eqssKraft} we obtain
$V^{-1}(F(D_{1,1})) \cap N_1 = D_{1,1}$
for the second summand of Equation~\eqref{eq:2x0downsummands}. We now compute the first summand using \cref{lem:SOM2x0}.
Assume $l > r_3$, so that $l$ exceeds at least three of the elements of $S_{\text{down}}$. We show 
\begin{equation} \label{eq:2x0downmoving}
V^{-1}(F(C_{1,l})) \cap M_1 \subseteq C_{1,l-1}.
\end{equation}
First, applying $F$ yields
$$F(C_{1,l}) = \begin{cases}
    C_{2,l-1} &\hbox{if $l <v$,} \\
    C_{2,l-2} &\hbox{if $l \geq v$}.
\end{cases}$$
In the first case, it follows that $l> q_M-u+1$ and hence
$$V^{-1}(F(C_{1,l})) \cap M_1 = V^{-1}(C_{2,l-1}) \cap M_1 = C_{1,l-1}.$$
In the second case, we use $l>q_M-v+2$ to obtain
$$V^{-1}(F(C_{1,l})) \cap M_1 = V^{-1}(C_{2,l-2}) \cap M_1 \subseteq C_{1,l-1}.$$
In both cases Equation~\eqref{eq:2x0downmoving} is satisfied.
Thus applying $V^{-1}F$ decreases the index in $C_{1,\bullet}$ for $l>r_3$. We show it cannot decrease further, meaning
\begin{equation} \label{eq:2x0downnotmoving}
V^{-1}(F(C_{1,r_3})) \cap M_1 = C_{1,r_3}.
\end{equation}
We begin by applying $F$:
$$ F(C_{1,r_3}) = \begin{cases} 
C_{2,r_3} &\hbox{if $r_3 = u-1$,} \\
C_{2,r_3-1} &\hbox{if $u \leq   r_3  < v$,} \\
C_{2,r_3-2} &\hbox{if $r_3 \geq v$}.
\end{cases}$$
In the case $r_3=u-1$, it follows that $r_3 \geq q_M-u+1$ and therefore
$$V^{-1}(F(C_{1,r_3})) \cap M_1 = V^{-1}(C_{2,r_3}) \cap M_1 = C_{1,r_3},$$
as desired. In the second case we have $u  \leq r_3 < v$ and it follows that $q_M-v+2 \leq r_3 \leq q_M-u+1$. Hence
$$V^{-1}(F(C_{1,r_3})) \cap M_1 = V^{-1}(C_{2,r_3-1}) \cap M_1 = C_{1,r_3}.$$
Finally, in the third case we have $r_3>v$. It follows that $r_3 = q_M-v+2$ and thus
$$V^{-1}(F(C_{1,r_3})) \cap M_1 = V^{-1}(C_{2,r_3-2}) \cap M_1 = C_{1,r_3}.$$
In each case, Equation~\eqref{eq:2x0downnotmoving} holds, showing that we cannot `move down' further down than $C_{1,r_3}$. This finishes the proof of the lemma.
\end{proof}

In order to describe the Ekedahl-Oort stratum of $L$, one more lemma is needed.

\begin{lemma} \label{lem2x0done}
We have $C_{1,r_2} \cap M[F]= C_{1,r_3} \cap M[F]$. 
\end{lemma}
\begin{proof}
Since $r_2 \leq r_3$ and, hence, $C_{1,r_2} \subseteq C_{1,r_3}$ the inclusion ``$\subseteq$''  is immediate, so we focus on the inclusion ``$\supseteq$''. Recall that $M_1[F]=\langle e_{1,u}, e_{1,v} \rangle$. Given the definitions of $S_{\textup{up}}$ and $S_{\textup{down}}$ one can check that it is not possible to have $r_2 < u < r_3$ or $r_2 < v < r_3$. Therefore there is no element of $M[F]$ gained when moving from $C_{1,r_2}$ to $C_{1,r_3}$. In other words, we have $C_{1,r_3} \cap M[F] \subseteq C_{1,r_2}$, which proves the lemma.
\end{proof} 

\subsubsection{Ekedahl-Oort strata}
Recall the setting. $M$ is the standard object of $\mathcal{M}(m,2)_{\gamma_{u,v}}$, and $C_{i,\bullet}$ are filtrations of $M_i$ coming from a final filtration on $M$. $N$ is the standard object of the unique Ekedahl-Oort stratum of $\mathcal{M}(1,0)$, with the filtration $D_{i,\bullet}$ of $N_i$. We form $L:=M \oplus N^n$, with the filtrations $E_{1,\bullet}$ of $L_i$. We define
\begin{align*}
    \eta^M_1(j)&=\dim (C_{1,j} \cap M[F]), \text{ and}\\
    \eta^L_1(j) &= \dim (E_{1,j} \cap L[F]). 
\end{align*}
Recall that $r_2 = \min(S_{\text{up}}) \setminus \set{\min S_{\text{up}}}$, the second smallest element of the multiset $S_{\textup{up}} := \{u,v,q_M-v+2,q_M-u+1\}$. We now completely describe the Ekedahl-Oort stratum of $L$.

\begin{theorem} \label{thm2x0}
The Ekedahl-Oort stratum of $L:=M \oplus N^n$ in $\mathcal{M}(m+n,2)$ is represented by the permutation $\gamma_{y,s}$, with
\[y = \begin{cases}
u &\hbox{if $u \leq r_2$,} \\
u+n &\hbox{if $u>r_2$,} \end{cases}\qquad\qquad
s = \begin{cases}
v &\hbox{if $v \leq r_2$,} \\
v+n &\hbox{if $v > r_2 $}. \end{cases}
\]

\end{theorem}

\begin{proof}
To prove the result, we compute the function $\eta^L_1$. From \cref{lem:2x0up} and \cref{lem:2x0down} we see that 
\begin{equation} \label{eq:E1j}
E_{1,j} = \begin{cases}
C_{1,j} \oplus D_{1,0}^n &\hbox{if $j\leq r_2$}, \\
C_{1,j-n} \oplus D_{1,1}^n &\hbox{if $j \geq r_3+n$}.
\end{cases}
\end{equation}
Hence, by Equation~\eqref{eqssKraft}, $\eta^L_1(j)=\eta^M_1(j)$ in the first case and $\eta^L_1(j)=\eta^M_1(j-n)$ in the second case. Furthermore, \cref{lem2x0done} implies that $\eta_1^L$ stays constant between $r_2$ and $r_3$. In summary:
\begin{equation} \label{eq:2x0tildeeta}
\eta^L_1(j) = \begin{cases}
\eta^M_1(j) &\hbox{if $j\leq r_2$}, \\
\eta^M_1(r_2) &\hbox{if $r_2< j <r_3+n$}, \\
\eta^M_1(j-n) &\hbox{if $j \geq r_3+n$}.
\end{cases}
\end{equation}

Note that $\eta^M_1$ jumps at $u$ and $v$. If $u\leq r_2$, then we are in case 1 of Equation \eqref{eq:2x0tildeeta} and $\eta^L_1$  jumps at $j=u$. If $u>r_2$ then, as noted in the proof of \cref{lem2x0done}, $u$ must also be greater than $r_3$. Hence, we are in case 3 of Equation \eqref{eq:2x0tildeeta} and $\eta^L_1$ jumps at $j=u+n$. Similarly, $\eta^L_1$ jumps at either $v$ or $v+n$ if $v\leq r_2$ or $v>r_2$, respectively. This results in the statement of the theorem.
\end{proof} 

This theorem has the following implication for supersingular loci.

\begin{corollary} \label{cor:2x0sslocus}
Let $m+n=q-2$ and $\mathcal{M}(m,2)_{\gamma_{u,v}} \cap \mathcal{M}(m,2)^{ss} \neq \emptyset$. Then we have $$\mathcal{M}(q-2,2)_{\gamma_{y,s}} \cap \mathcal{M}(q-2,2)^{ss} \neq \emptyset,$$ where $y$ and $s$ are computed from $u$ and $v$ as in \cref{thm2x0}.
\end{corollary}
\begin{proof}
This follows immediately from \cref{thm2x0}, the fact $\mathcal{M}(n,0)=\mathcal{M}(n,0)^{ss}$, and the observation that the product of two supersingular abelian varieties is supersingular.
\end{proof}

Note that \cref{cor:2x0sslocus} allows one to gain information about some Ekedahl-Oort strata that intersect $\mathcal{M}(q-2,2)^{ss}$, using information about which Ekedahl-Oort strata intersect $\mathcal{M}(m,2)^{ss}$, for $m$ smaller than $q-2$. Fortunately, there is a complete classification of Ekedahl-Oort strata that intersect the supersingular locus of $\mathcal{M}(m,2)$ for very small $m$. We first recall results for $m=0$ and 1.

The Shimura variety $\mathcal{M}(0,2)$ is isomorphic to $\mathcal{M}(2,0)$. As mentioned in \cref{subsubsec:gen2x0}, $\mathcal{M}(0,2)$ has a unique Ekedahl-Oort stratum. This must coincide with the supersingular locus, and so $\mathcal{M}(0,2) = \mathcal{M}(0,2)^{ss}$.

The Shimura variety $\mathcal{M}(2,1)$ is isomorphic to $\mathcal{M}(1,2)$, and the interaction between the Ekedahl-Oort strata and the supersingular locus in this case is summarized in \cref{subsubsec:gen1X1}.

\begin{proposition}\label{HP22}
The supersingular locus of $\mathcal{M}(2,2)$ is exactly the union of those Ekedahl-Oort strata of dimension two or less:
$$\mathcal{M}(2,2)^{ss} = \mathcal{M}(2,2)_{\gamma_{1,2}} \sqcup  \mathcal{M}(2,2)_{\gamma_{1,3}} \sqcup  \mathcal{M}(2,2)_{\gamma_{1,4}} \sqcup  \mathcal{M}(2,2)_{\gamma_{2,3}}. $$
\end{proposition}

\begin{proof}
By \cite{goertzhe}, the Shimura variety $\mathcal{M}(2,2)$ is of Coxeter type. In particular, the supersingular locus is a union of Ekedahl-Oort strata since every Ekedahl-Oort stratum intersecting the supersingular locus is entirely contained in the supersingular locus.

There are six Ekedahl-Oort strata of $\mathcal{M}(2,2)$, including exactly two strata of dimension two: those indexed by $\gamma_{2,3}$ and $\gamma_{1,4}$. It is known by results of \cite{howard2014supersingular} that the supersingular locus $\mathcal{M}(2,2)^{ss}$ is two-dimensional, and so at least one of $\mathcal{M}(2,2)_{\gamma_{1,4}}$ or $\mathcal{M}(2,2)_{\gamma_{2,3}}$ must be contained in the supersingular locus. In what follows, we observe that, in fact, both are.

There is a natural automorphism of $\mathcal{M}(2,2)(\kk)$, taking $(A, \iota, \lambda, \xi)$ to $(A, \overline{\iota}, \lambda, \xi)$, where $\overline{\iota}$ is the action $\iota$ composed with the nontrivial automorphism of $K$. This action stabilizes the supersingular locus, but has the effect of conjugating the action on the $p$-torsion group schemes defining the Ekedahl-Oort strata. On the level of the standard objects $N_{\gamma_{u,v}}$ corresponding to $\gamma_{u,v}$, this action interchanges the roles of the basis vectors $e_{1,j}$ and $e_{2,j}$ for each $1 \leq j \leq q$. 

By applying \cref{lem:SOuv} to compute the standard objects $N_{1,4}$ and $N_{2,3}$, one can see immediately that interchanging the roles of $e_{1,j}$ and $e_{2,j}$ defines an isomorphsim between $N_{1,4}$ and $N_{2,3}$. Therefore, both $\mathcal{M}(2,2)_{\gamma_{1,4}}$ and $\mathcal{M}(2,2)_{\gamma_{2,3}}$ are contained in the supersingular locus.

Since $\mathcal{M}(2,2)_{\gamma_{1,4}}$ and $\mathcal{M}(2,2)_{\gamma_{2,3}}$ are contained in the supersingular locus and the supersingular locus is closed, the closure of these strata are also contained in the supersingular locus. By \cref{prop:level1Bruhat}, 
$$\overline{\mathcal{M}(2,2)_{\gamma_{1,4}} \sqcup \mathcal{M}(2,2)_{\gamma_{2,3}}} = \mathcal{M}(2,2)_{\gamma_{1,2}} \sqcup  \mathcal{M}(2,2)_{\gamma_{1,3}} \sqcup  \mathcal{M}(2,2)_{\gamma_{1,4}} \sqcup  \mathcal{M}(2,2)_{\gamma_{2,3}},$$
i.e., exactly those strata of dimension less than or equal to two. 

Finally, the strata indexed by $\gamma_{2,4}$ and $\gamma_{3,4}$ have dimension greater than two, and so they cannot be contained in the supersingular locus. 
\end{proof}


\section{The Forgetful Map: Relation to Siegel Modular Variety}\label{sec:smv}

In the previous section we obtained information about the Shimura variety $\mathcal{M}(q-2,2)$ via product maps. The results of that section, particularly \cref{cor:1x1sslocus,cor:2x0sslocus}, allow us to draw conclusions about the interaction between Ekedahl-Oort strata and Newton strata in some cases. In this section, we further our study of this interaction by relating our Shimura variety $\mathcal{M}(q-2,2)$ to the Siegel modular variety $\mathcal{A}_q$. 

\subsection{Background} \label{subsec:SMVbackground}
By ``forgetting" the unitary structure of the abelian varieties we are considering, we obtain a point on the Siegel modular variety $\mathcal{A}_q$. In other words, there is a \textbf{forgetful map}
\begin{align*}
    \Psi: \mathcal{M}(q-2,2) &\to \mathcal{A}_q \\
    (A, \lambda, \iota, \xi) &\mapsto (A,\lambda,\xi).
\end{align*}
The map maintains the polarisation $\lambda$ and the level structure $\xi$ of the abelian variety $A$, but drops the unitary structure $\iota: \mathcal{O}_K \to \text{End}(A)$ of signature $(q-2,2)$.  This forgetful map induces a map on Ekedahl-Oort strata
$$\psi:\mathbf{W}(q-2,2) \to \mathbf{W}_q,$$
where $\mathbf{W}_q$ is the set of minimal length Weyl group coset representatives indexing the Ekedahl-Oort strata of $\mathcal{A}_q$. By \cite[3.6]{Moonen2001}, the set $\mathbf{W}_q$ consists of the permutations $\omega \in \mathfrak{S}_{2q}$ satisfying $\omega^{-1}(1) < \omega ^{-1}(2) < \cdots < \omega^{-1}(q)$ and $\omega(i) + \omega(2q+1-i) = 2q+1$. This section is devoted to computing $\psi(\gamma_{u,v})$ for $\gamma_{u,v} \in \mathbf{W}(q-2,2)$ with the goal of obtaining information about $\mathcal{M}(q-2,2)$ from various results about $\mathcal{A}_q$.

We now outline how, given $\gamma_{u,v} \in \mathbf{W}(q-2,2)$, the permutation $\psi(\gamma_{u,v}) \in \mathbf{W}_q$ is computed. Using the bijection in \cref{thm:Moonenbij} and \cref{lem:SOuv}, we construct the standard object $M$ corresponding to the Ekedahl-Oort stratum $\mathcal{M}(q-2,2)_{\gamma_{u,v}}$. Upon forgetting the $\mathbx{F}_{p^2}$-action of $M$, we observe that $M$ is the Dieudonn\'{e} module of an Ekedahl-Oort stratum of $\mathcal{A}_q$. By \cite[3.6]{Moonen2001}, this Ekedahl-Oort stratum corresponds to an element of $\mathbf{W}_q$. This correspondence works as follows. Let $W_\bullet$ be a final filtration of $M$ and define the non-decreasing function
$$\eta(j):= \dim(W_j \cap M[F])$$
for $0 \leq j \leq 2q$. Note that $\eta(2q)= \dim(M[F]) = q$. Let $1 \leq j_1<  \cdots < j_q \leq 2q$ be the indices where $\eta$ jumps, meaning $\eta(j_l) = \eta (j_l -1) +1$, and let $1 \leq i_1 < \cdots < i_q \leq 2q$ be the remaining indices. We define $\omega_{u,v} \in \mathfrak{S}_{2q}$ by $\omega_{u,v}(j_l)=l$ and $\omega_{u,v} (i_m)=m+q$. By construction, $\omega_{u,v}$ is an element of $\mathbf{W}_q$ and $\psi(\gamma_{u,v})=\omega_{u,v}$.

The duality coming from the symplectic pairing on $M$ implies that $\eta(j)=\eta(j-1)$ holds if and only if $\eta(2q+1-j)=\eta(2q-j)+1$ holds. By induction, it follows that 
$\eta(j) + q = \eta(2q-j) + j.$
The identity $i_m=2q+1-j_{q-m}$ then follows from induction, implying that
\begin{equation} \label{eqomegaduality}
\omega_{u,v}(i) + \omega_{u,v}(2q+1-i) = 2q+1.
\end{equation}

The preceding paragraphs show that the function $\eta(j)=\dim(W_j \cap M[F])$ determines the permutation $\omega_{u,v}=\psi(\gamma_{u,v})$, and so we focus on determining all the values of $\eta$. \cref{lem:SOuv} yields the description  
$$M[F]= \langle e_{1,u} , e_{1,v} \rangle \oplus \langle e_{2,j} \; | \; j \notin \{q+1-u,q+1-v\} \rangle.$$
Recall the filtrations $C_{i,\bullet} = W_\bullet \cap M_i$ and note that any $W_j$ can be written as $W_j = C_{1,l_1} \oplus C_{2,l_2}$ for some $l_1$ and $l_2$ satisfying $l_1 + l_2 = j$. In \cref{prop:Siegelcanonicalfiltration}, this decomposition is made explicit, and as a result a final filtration of $M$ is given in terms of $C_{1,\bullet}$ and $C_{2,\bullet}$. The shape of this final filtration depends on $u$ and $v$. In \cref{thm:SiegelWeyl}, this final filtration is used to determine the permutation $\omega_{u,v} \in \mathbf{W}_q$.

There are three possibilities for the Ekedahl-Oort stratum of $(\mathcal{A}_{q})_{\omega_{u,v}}$. The first possibility is that $(\mathcal{A}_q)_{\omega_{u,v}}$ is contained in $\mathcal{A}_q^{ss}.$ In that case, it follows that $\mathcal{M}(q-2,2)_{\gamma_{u,v}}$ is contained in $\mathcal{M}(q-2,2)^{ss}.$ \cref{cor:ruledinbySMV} records the Ekedahl-Oort strata of $\mathcal{M}(q-2,2)$ that occur in this way. Another possibility is that $(\mathcal{A}_q)_{\omega_{u,v}}$ is disjoint from $\mathcal{A}_q^{ss}.$ In that case, it follows that $\mathcal{M}(q-2,2)_{\gamma_{u,v}}$ is disjoint from $\mathcal{M}(q-2,2)^{ss}$. \cref{prop:ruledoutbyFVnilp} and \cref{cor:ruledoutbyminimal} record the Ekedahl-Oort strata of $\mathcal{M}(q-2,2)$ with this property. Finally, it is possible that the Ekedahl-Oort stratum $(\mathcal{A}_q)_{\omega_{u,v}}$ intersects $\mathcal{A}_q^{ss}$, but is not contained in it. In this case, we do not obtain information about the Ekedahl-Oort stratum $\mathcal{M}(a,b)_{\gamma_{u,v}}$, due to the highly non-surjective nature of the forgetful map. It could, for instance, be that there exist supersingular abelian varieties in $(\mathcal{A}_q)_{\omega_{u,v}}$, but that none of them admit a compatible unitary structure of signature $(q-2,2)$. For this reason, our work in this section focuses on the cases when the Ekedahl-Oort $(\mathcal{A}_q)_{\omega_{u,v}}$ is either completely contained in $\mathcal{A}_q^{ss}$ or completely disjoint from it.

\subsection{A Final Filtration of the Standard Object} \label{subsec:SMVfinalfiltration}

In order to compute $\eta(j) = \dim(W_j \cap M[F])$, we construct a final filtration $W_\bullet$ of $M$ in terms of $C_{1,\bullet}$ and $C_{2,\bullet}$. Essentially, this is done by applying \cref{lem:SOuv} (or its consequence \cref{lem:SOM2x0}) repeatedly. We could give a canonical filtration instead, but giving a final filtration makes the exposition shorter and later computations easier. It is sufficient to construct $W_j$, where $1 \leq j \leq q$, as the remaining parts of the filtration can be constructed by taking symplectic complements. As seen in the proof of \cref{lem:F^2}, we have
$$ W_j^\perp = (C_{1,l_1} \oplus C_{2,l_2})^\perp = C_{1,l_1}^\perp \cap C_{2,l_2}^\perp = (M_1 \oplus C_{2,q-l_1}) \cap (C_{1,q-l_2} \oplus M_2) = C_{1,q-l_2} \oplus C_{1,q-l_1}= W_{2q-j}.$$

The following proposition gives a final filtration $W_\bullet$ of $M$ that is dependent on $u$ and $v$. When there is only one way to fill up a gap between two subspaces, for instance between $C_{1,1} \oplus C_{2,0}$ and $C_{1,1} \oplus C_{2,q-2}$, then the subspaces in between (of the form $C_{1,1} \oplus C_{2,l}$ for $1<l<q-2$) are omitted from the notation. In each case, \cref{lem:SOuv} yields
\begin{align}
W_q &= F(M)= \langle e_{1,1}, e_{1,2} \rangle \oplus \langle e_{2,j} \; | \; j \leq q-2 \rangle =C_{1,2} \oplus C_{2,q-2}, \nonumber \\
M[F] &= M[F]_1 \oplus M[F]_2 = \langle e_{1,u}, e_{1,v} \rangle \oplus \langle e_{2,j} \; |  \; j \notin \{q+1-v,q+1-u\} \rangle. \label{eq:M[F]}
\end{align}

\begin{proposition} \label{prop:Siegelcanonicalfiltration}
Let $M$ be the standard object of $\mathcal{M}(q-2,2)_{\gamma_{u,v}}$. In the case $u=1$, the first half of a final filtration of $M$ is 
$$0 \subset C_{1,1} \oplus C_{2,0} \subset \cdots \subset C_{1,1} \oplus C_{2,q-2} \subset W_q.$$
In the case $u=2$, the first half of a final filtration is
$$ 0 \subset C_{1,0} \oplus C_{2,1} \subset C_{1,1} \oplus C_{2,1} \subset \cdots \subset C_{1,1} \oplus C_{2,s_1} \subset C_{1,2} \oplus C_{2,s_1} \subset \cdots \subset W_q.$$
In the case $u > 2$ and $v<q-1$, the first half of a final filtration is
$$ 0 \subset \cdots \subset C_{1,0} \oplus C_{2,2} \subset \cdots \subset C_{1,2} \oplus C_{2,2} \subset \cdots \subset W_q.$$
In the case $u > 2$ and $v \geq q-1$, the first half of a final filtration is
$$0 \subset C_{1,0} \oplus C_{2,1} \subset C_{1,1} \oplus C_{2,1} \subset C_{1,1} \oplus C_{2,2} \subset C_{1,2} \oplus C_{2,2} \subset \cdots \subset W_q.$$
\end{proposition}
\begin{proof}
We prove the proposition by constructing the necessary parts of the canonical filtration. 
The standard object from \cref{lem:SOuv} is used for determining the effect on $F$ or $V^{-1}$ on a subspace. We split the proof up into three cases: $u=1$, $u=2$ and $u > 2$.

First, assume $u=1$. In that case, we construct
\begin{equation*}
F(W_q) = \begin{cases} 
    C_{1,0} \oplus C_{2,0} = W_0 &\hbox{if $v=2$},  \\
    C_{1,1} \oplus C_{2,1} = W_2 &\hbox{if $v>2$}. \end{cases}
\end{equation*}
In the case $(u,v)=(1,2)$, we have $W_q=M[F]$ and therefore $\eta(q)=q$, implying that $\eta$ has to increase by 1 at each index between $0$ and $q$. It is not possible to construct a subspace between $0$ and $W_q$ using $F$ and $V^{-1}$, so we can extend the canonical filtration to a final filtration in any way we like; it will not influence the function $\eta$. In the case $v>2$, we construct
$$F(W_2)=\begin{cases} 
    C_{1,0} \oplus C_{2,0}  = W_0 &\hbox{if $v<q$}, \\
    C_{1,1} \oplus C_{2,0}  = W_1 &\hbox{if $v=q$}.
\end{cases}$$ 
Here we have used that $F(C_{1,1})=0$, since $u=1$, and $F(C_{2,1})=0$ if and only if $v<q$. In the first case, the canonical filtration cannot produce $W_1$. This is not required, as $\eta(2)=2$ implies $\eta(1)=1$. In both cases, we can form a final filtration containing $W_1$.

Next, we fill in the canonical filtration between $W_2$ and $W_q$. This amounts to discovering when $C_{1,1}$ changes to $C_{1,2}$ in the filtration. We construct
$$V^{-1}(W_q) = C_{1,q-1} \oplus C_{2,q-1} = W_{2q-2} = W_2^\perp.$$ 
We then apply $F$ to obtain
$$F(W_{2q-2})= \begin{cases}
C_{1,1} \oplus C_{2,q-2} = W_{q-1} &\hbox{if $v=q$}, \\
C_{1,1} \oplus C_{2,q-3} = W_{q-2} &\hbox{if $v<q$}.
\end{cases}$$ 
Thus far, we have constructed $W_0$, $W_1$, $W_2$, $W_{q-1}$, and $W_q$. This forces all the intermediate $W_i = C_{1,1}\oplus C_{2,i-1}$ due to dimension reasons, as we showed that all have the $C_{1,1}$ term in the first half. In the case $v=q$, we are done: the first half of the canonical filtration is 
\begin{equation} \label{eqfinalfiltu=1}
0 \subset C_{1,1} \oplus C_{2,0} \subset C_{1,1} \oplus C_{2,1} \subset \cdots \subset C_{1,1} \oplus C_{2,q-2} \subset C_{1,2} \oplus C_{2,q-2}=W_q.
\end{equation} 
In the case $v=3$, we compute 
$$\eta(q-2) = \dim\left( W_{q-2} \cap M[F] \right) = \eta_1(1) + \eta_2(q-3) = 1+(q-3) = q-2.$$
Together with $\eta(q)=q-2$, this implies $\eta(q-1)=q-2$, meaning that $W_{q-1}$ does not need to be constructed. The final filtration in Equation \eqref{eqfinalfiltu=1} is still valid in this case.

In the case $3<v<q$, one more step is needed. We construct
\begin{align*}
V^{-1}(W_{2}) &= V^{-1}(C_{1,1} \oplus C_{2,1}) = C_{1,3} \oplus C_{2,q-1} = W_{q+2}, \\
V^{-1}(W_{q+2}) &= V^{-1}(C_{1,3} \oplus C_{2,q-1}) = C_{1,q} \oplus C_{2,q-1} = W_{2q-1}, \\
F(W_{2q-1}) &= F(C_{1,q} \oplus C_{2,q-1}) = C_{1,1} \oplus C_{2,q-2} = W_{q-1}.
\end{align*}
We conclude that the final filtration given in Equation~\eqref{eqfinalfiltu=1} always works in the case $u=1$.

We now treat the case $u=2$, which is the most challenging. We again construct
$$F(W_q)=C_{1,1} \oplus C_{2,1} = W_2.$$ 
In the case $v=q$, we compute $\eta(2)=0$. Combining this with $\eta(q)=q-2$ fixes the behavior of $\eta$ everywhere; it has to jump everywhere between $\eta(2)=0$ and $\eta(q)=q-2$. Therefore we are finished with this case and assume $2<v<q$ from here on. Applying $F$ to $W_2$ yields
$$F(W_2) = C_{1,0} \oplus C_{2,1}  = W_1.$$ 
This implies that in the case $u=2$ the final filtration must be of the form
\begin{align*}
    0 \subset C_{1,0} \oplus C_{2,1} \subset C_{1,1} \oplus C_{2,1} &\subset C_{1,0} \oplus C_{2,2} \subset C_{1,1} \oplus C_{2,2} \subset \cdots\\
    &\cdots \subset C_{1,1} \oplus C_{2,m} \subset C_{1,2} \oplus C_{2,m} \oplus \cdots \subset C_{1,2} \oplus C_{2,q-2} = W_q
\end{align*}
for some (not necessarily unique) integer $m$. The remaining goal is to find such an integer $m$. We do this by constructing $C_{1,1} \oplus C_{2,j}$ for increasing $j$ (by `moving up') and constructing $C_{1,2} \oplus C_{2,j}$ for decreasing $j$ (by `moving down'). 

We first outline the moving down procedure. Define $s_2:=\max \{v-2,q-v\}$. Starting from $W_q=C_{1,2} \oplus C_{2,q-2}$, we construct $W_{j+2} = C_{1,2} \oplus C_{2,j}$ for decreasing $j$, by applying $FV^{-1}$, until we reach $C_{1,2} \oplus C_{2,s_2}$. The induction step goes as follows. Assume $j > s_2$. Since $j > q-v$, we have
$$ V^{-1}(W_{j+2}) = V^{-1}(C_{1,2} \oplus C_{2,j})= C_{1,j+1} \oplus C_{2,q-1} = W_{q+j}.$$ 
As additionally $j+1 \geq v$ holds, we have
$$F(C_{1,j+1} \oplus C_{2,q-1}) = C_{1,2} \oplus C_{2,j-1}= W_{j+1}.$$ 
In this way $C_{1,2} \oplus C_{2,j}$ is constructed for decreasing $j$. We now show that this process halts when $j=s_2$ is reached. If $j=q-v$, then we have 
$$ V^{-1}(C_{1,2} \oplus C_{2,j})= C_{1,j+2} \oplus C_{2,q-1} = W_{q+j+1}.$$ 
and we cannot move down further. Similarly, if $j=v-2$, then we obtain
$$F(V^{-1}(C_{1,2} \oplus C_{2,j})) = F(C_{1,j+1} \oplus C_{2,q-1}) = C_{1,2} \oplus C_{2,j}.$$ 
We conclude that we can move down to $W_{2+s_2}= C_{1,2} \oplus C_{2,s_2}$ and no further.

The moving up procedure works analogously. Define $s_1:=\min \{v-2,q+1-v\}$. We begin with $W_2=C_{1,1} \oplus C_{2,1}$ and construct $W_{j+1}=C_{1,1} \oplus C_{2,j}$ for increasing $j$, by applying $FV^{-1}$, until we reach $C_{1,1} \oplus C_{2,s_1}$. Now, for the induction step we assume $j<s_1$. Since $j\leq q-v$, we have
$$V^{-1} (W_{j+1}) = V^{-1} (C_{1,1} \oplus C_{2,j}) = C_{1,j+2} \oplus C_{2,q-2}=W_{q+j}.$$ 
Then, as $j+2 <v$ holds, we obtain 
$$F(W_{q+j}) = F(C_{1,j+2} \oplus C_{2,q-2}) = C_{1,1} \oplus C_{2,j+1} = W_{j+2}.$$ 
Thus $C_{1,1} \oplus C_{2,j}$ is constructed for increasing $j$. We now show that this procedure stops when $j=s_1$ is reached. In the case $j=q+1-v$, we have 
$$V^{-1}(W_{j+1}) = V^{-1} (C_{1,1} \oplus C_{2,j}) = C_{1,j+1} \oplus C_{2,q-2} = W_{q+j-1}.$$ 
On the other hand, in the case $j=v-2$, we have
$$F(W_{q+j}) = F(C_{1,j+2} \oplus C_{2,q-2}) = C_{1,1} \oplus C_{2,j}.$$ 

Finally, it is left to prove that this construction suffices, in the sense that it is not necessary to construct more parts of the filtration in order to compute the function $\eta$. In general there may be subspaces between $W_{1+s_1}$ and $W_{2+s_2}$. We now show that this does not influence the behavior of $\eta$. 

In the case $v-2 \geq q+1-v$, we have $W_{1+s_1} = W_{q+2-v}= C_{1,1} \oplus C_{2,q+1-v}$. Since $e_{2,q+1-v} \notin M[F]$, we obtain that $\eta(q+2-v)=q-v$. Then $\eta$ must keep increasing till $\eta(q)=q-2$, so the canonical filtration does not develop further.

On the other hand, in the case $v-2<q+1-v$ we have $W_{2+s_2} = W_{q+2-v} = C_{1,2} \oplus C_{2,q-v}$. We compute and obtain $\eta(q+2-v)=q+1-v$. This implies that $\eta$ must keep increasing between $\eta(2)=1$ and $\eta(q+2-v)=q+1-v$. Hence the canonical filtration does not develop further. In both cases the first half of a final filtration of $M$ is given by
$$ 0 \subset C_{1,0} \oplus C_{2,1} \subset C_{1,1} \oplus C_{2,1} \subset \cdots \subset C_{1,1} \oplus C_{2,s_1} \subset C_{1,2} \oplus C_{2,s_1} \subset \cdots \subset W_q.$$ 

The last case to treat is $u>2$. Here we obtain
$$F(W_q) = C_{1,2} \oplus C_{2,2}=W_4,$$ 
which already fixes the final filtration between $W_4$ and $W_q$: everything must be of the form $C_{1,2} \oplus C_{2,j}$ for $2 \leq j \leq q-2$. The only task at hand is to compute $\eta(1)$, $\eta(2)$ and $\eta(3)$. In the case $u=q-1$, we have $F(W_4)=W_4$ and therefore $\eta(4)=0$, which implies that $\eta(1)=\eta(2)=\eta(3)=0$. We now assume $u<q-1$ from here on and compute
$$F(W_4) = \begin{cases} 
    C_{1,0} \oplus C_{2,2} = W_2 &\hbox{if $2 \leq q-v$}, \\
    C_{1,1} \oplus C_{2,2} = W_3 &\hbox{if $2 > q-v.$}  \\
\end{cases}$$ 

First, we treat the case $2\leq q-v$. From the expression for $W_2$ and the fact that a final filtration is a filtration, it follows what $W_1$ and $W_3$ must be. This gives the final filtration
$$ 0 \subset \cdots \subset C_{1,0} \oplus C_{2,2} \subset \cdots \subset C_{1,2} \oplus C_{2,2} \subset \cdots \subset W_q. $$

We finally treat the case $2 > q-v$ or equivalently $v \geq q-1$. Applying $F$ again gives
\begin{align*}
    F(W_3) &= C_{1,1} \oplus C_{2,1} = W_2 , \text{ and }\\
    F(W_2) &= \begin{cases} 
    C_{1,0} \oplus C_{2,1}  = W_1 &\hbox{if $v=q-1$},\\
    C_{1,1} \oplus C_{2,1}  = W_2 &\hbox{if $v=q.$} \\
\end{cases}
\end{align*} 
When $v=q$, we have $\eta(2)=0$ and therefore $\eta(1)=0$. On the other hand, in the case $v=q-1$ we have constructed $W_1$, $W_2$ and $W_3$, which clearly suffices. In both cases, there is a final filtration
\[\pushQED{\qed}
0 \subset C_{1,0} \oplus C_{2,1} \subset C_{1,1} \oplus C_{2,1} \subset C_{1,1} \oplus C_{2,2} \subset C_{1,2} \oplus C_{2,2} \subset \cdots \subset W_q. \qedhere
\popQED\]
\renewcommand{\qedsymbol}{}
\end{proof}

\subsection{Weyl group cosets}\label{sec:forgetfulWeylGroupCoset}

We now compute the permutation $\omega_{u,v}=\psi(\gamma_{u,v}) \in \mathbf{W}_q \subset \mathfrak{S}_{2q}$ using the results of \cref{prop:Siegelcanonicalfiltration}. The resulting permutation represents the Ekedahl-Oort stratum of $M$ in $\mathcal{A}_q$. 

In the following theorem, only the action of $\omega_{u,v}$ on the integers $1 \leq i \leq q$ is given, as the remaining information can be retrieved using Equation \eqref{eqomegaduality}.

\begin{theorem}\label{thm:SiegelWeyl}
In the case $(u,v)=(1,2)$, we have $\omega_{u,v}= \mathrm{id}$. For various other choices of $u,v$, the following tables describe the corresponding $\omega_{u,v}$.

{\renewcommand{\arraystretch}{1.75}
\begin{center}
\begin{tabular}{|c|c|}
\hline
& $u = 1$ and $v > 2$\\
\hline
$\omega_{u,v}(i)$ & $\begin{cases}
i &\hbox{if $i<q+2-v$,} \\
q+1 &\hbox{if $i=q+2-v$,} \\
i-1 &\hbox{if $q+2-v<i<q$,} \\
q+2 &\hbox{if $i=q$}.
\end{cases}$\\
\hline
\end{tabular}
\end{center}

\begin{center}
\begin{tabular}{|c|c|c|c|}
\hline
& $u=2$ and $v=q$ 
& $u=2$ and $1<q+1-v \leq v-2$ 
&  $u=2$ and $q+1-v > v-2$  \\
 \hline
$\omega_{u,v}(i)$

& $\begin{cases}
q+i &\hbox{if $i\leq 2$,} \\
i-2 &\hbox{if $2 < i \leq q$}.
\end{cases}$ 
&
$\begin{cases}
1 &\hbox{if $i=1$,} \\
q+1 &\hbox{if $i=2$,} \\
i-1 &\hbox{if $2<i<q+2-v$,} \\
q+2 &\hbox{if $i=q+2-v$,} \\
i-2 &\hbox{if $q+2-v < i \leq q.$}
\end{cases}$
&
$\begin{cases}
1 &\hbox{if $i=1$,} \\
q+1 &\hbox{if $i=2$,} \\
i-1 &\hbox{if $2<i<q+3-v$,} \\
q+2 &\hbox{if $i=q+3-v$,} \\
i-2 &\hbox{if $q+3-v < i \leq q$.}
\end{cases}$
\\
\hline
\end{tabular}
\end{center}

 \makebox[\linewidth][c]{
 \begin{tabular}{|c|c|c|c|}
\hline
&$2 < u < q-1$ and $v=q-1$  & $2 < u <q-1$ and $v=q$ 
& $2 < u < q-1$ and $v<q-1$ 
\\
 \hline
$\omega_{u,v}(i)$
&
$\begin{cases}
1 &\hbox{if $i=1$,} \\
q+i-1 &\hbox{if $1<i\leq 4$,} \\
i-3 &\hbox{if $4<i< q+3-u$,} \\
q+4 &\hbox{if $i=q+3-u$,} \\
i-4 &\hbox{if $q+3-u < i \leq q$}.
\end{cases}$
& $\begin{cases}
q+i &\hbox{if $1 \leq i \leq 2$,} \\
i-2 &\hbox{if $i=3$,} \\
q+3 &\hbox{if $i=4$,} \\
i-3 &\hbox{if $4<i\leq q+3-u$,} \\
q+4 &\hbox{if $i=q+3-u$,} \\
i-4 &\hbox{if $q+3-u < i \leq q$}.
\end{cases}$
&
 $ \begin{cases}
i &\hbox {if $1 \leq i \leq 2$,} \\
q+i-2 &\hbox{if $3 \leq i \leq 4$,} \\
i-2 &\hbox{if $4 < i \leq q+3-v$,} \\
q+3 &\hbox{if $i=q+3-v$,} \\
i-3 &\hbox{if $q+3-v < i < q+3-u$,} \\
q+4 &\hbox{if $i=q+3-u$,} \\
i-4 &\hbox{if $q+3-u < i \leq q$}.
\end{cases}$

\\
\hline
\end{tabular}}

\begin{center}
\begin{tabular}{|c|c|}
\hline
& $u = q-1$\\
\hline
$\omega_{u,v}(i)$ & $\begin{cases}
q+i &\hbox{if $1 \leq i\leq 4$,} \\
i-4 &\hbox{if $4<i \leq q$}.
\end{cases}$\\
\hline
\end{tabular}
\end{center}
}
\end{theorem}

\begin{proof}
In the case $(u,v)=(1,2)$, we have $\eta(q)=q$. Therefore the function $\eta$ jumps at every integer $1 \leq j \leq q$. Hence we obtain $j_l=l$ and $i_m=m+q$. We conclude $\omega_{u,v}=\mathrm{id}$.

In the case $u=1$ and $v>2$, we follow the final filtration given in \cref{prop:Siegelcanonicalfiltration}. We compute
$$\eta(j)= \begin{cases}
1 &\hbox{if $j=1$}, \\
1+\eta_2(j-1) &\hbox{if $1<j <q$}, \\
1+\eta_2(j-2) &\hbox{if $j=q$}
\end{cases}$$
using the information we have on $\eta_1$ and $\eta_2$. We see that $\eta$ jumps everywhere except at the indices where $e_{1,2}$ and $e_{2,q+1-v}$ are added. The computation above shows where this happens: $i_1=q+2-v$ and $i_2=q$. This is the information needed to compute $\omega_{u,v}$:
$$ \omega_{u,v}(i) = \begin{cases}
i &\hbox{if $i<q+2-v$}, \\
q+1 &\hbox{if $i=q+2-v$}, \\
i-1 &\hbox{if $q+2-v<i<q$}, \\
q+2 &\hbox{if $i=q$}.
\end{cases}$$
as asserted.

In the case $u=2$, set $s_1=\min\{v-2,q+1-v\}$. Then by \cref{prop:Siegelcanonicalfiltration} it follows that
$$\eta(j)= \begin{cases}
\eta_2(j) &\hbox{if $j=1$}, \\
\eta_2(j-1) &\hbox{if $1<j \leq s_1+1$}, \\
1+\eta_2(j-2) &\hbox{if $s_1+1< j \leq q$}.
\end{cases}$$
The integers $i_1$ and $i_2$ correspond to the indices where $e_{1,1}$ and $e_{2,q+1-v}$ are added. In the case $v=q$, this implies $i_1=1$ and $i_2=2$. Otherwise, we have $i_1=2$. In the case $1< q+1-v \leq v-2$, we have $i_2=q+2-v$. In the case $q+1-v > v-2$, we have $i_2=q+3-v$. 

In the case $2<u<q-1$ and $v<q-1$, the final filtration in \cref{prop:Siegelcanonicalfiltration} implies
$$\eta(j)= \begin{cases}
\eta_2(j) &\hbox{if $j\leq 2$}, \\
\eta_2(j-1) &\hbox{if $j=3$}, \\
\eta_2(j-2) &\hbox{if $3 < j \leq q$}.
\end{cases}$$
Thus $\eta$ jumps everywhere except at the indices $\{3,4,q+3-v,q+3-u\}$, resulting in the permutation $\omega_{u,v}$ from the theorem.

Finally, in the case $u > 2$ and $v \geq q-1$, \cref{prop:Siegelcanonicalfiltration} provides
$$\eta(j)= \begin{cases}
\eta_2(j) &\hbox{if $j=1$}, \\
\eta_2(j-1) &\hbox{if $j \in \{2,3\}$}, \\
\eta_2(j-2) &\hbox{if $3< j \leq q$}.
\end{cases}$$
Now there are four indices where $\eta$ does not jump: these are the indices when $e_{1,1}$, $e_{1,2}$, $e_{2,q+1-u}$ and $e_{2,q+1-v}$ are added. In each case, $2$ and $4$ are among these indices. In the case $u=q-1$, the remaining indices are $1$ and $3$. In the case $u<q-1$ and $v=q$, the remaining indices are $1$ and $q+3-u$. Finally, in the case $u<q-1$ and $v=q-1$, the remaining indices are $3$ and $q+3-u$. 
\end{proof}

\subsection{Ekedahl-Oort strata contained in the supersingular locus}

Using the main result of \cite{Hoeve2009}, we pinpoint exactly which $\omega_{u,v}$ represent an Ekedahl-Oort stratum of $\mathcal{A}_q$ that is contained in $\mathcal{A}_q^{ss}$.

\begin{proposition} \label{prop:SMVcontained}
The Ekedahl-Oort stratum $(\mathcal{A}_q)_{\omega_{u,v}}$ is contained in $\mathcal{A}_q^{ss}$ if and only if $u=1$ and $v < \lfloor q/2 \rfloor+2$.
\end{proposition} 
\begin{proof}
By \cite[Theorem 1.2]{Hoeve2009}, the containment holds if and only if $\omega_{u,v}(i)=i$ for each $1 \leq i \leq \lceil q/2 \rceil$. For $u=1$, by \cref{thm:SiegelWeyl} this holds exactly when $\lceil q/2 \rceil <q+2-v$, i.e., when $v<\lfloor q/2 \rfloor + 2$. 

We now show that containment is not possible for $u>1$. In the case $u=2$, we have $\omega_{u,v}(2)>q$, so the containment can only hold if $q\leq 2$. That, however, contradicts the assumption $u=2$. 

In the case $u > 2$, we also have $\omega_{u,v}(2)>q$, except when $2 < u <q-1$ and $v <q-1$. In this case we have $\omega_{u,v}(3)>q$ and hence $\lceil q/2 \rceil<3$. This implies  $q \leq 4$, which contradicts $2< u < v < q-1 \leq 3 $.  
\end{proof}

\begin{corollary} \label{cor:ruledinbySMV}
Assume $u=1$ and $v<\lfloor q/2 \rfloor+2$. Then the Ekedahl-Oort stratum $\mathcal{M}(q-2,2)_{\gamma_{u,v}}$ is contained in $\mathcal{M}(q-2,2)^{ss}$.
\end{corollary}
\begin{proof}
\cref{prop:SMVcontained} shows that $\omega_{u,v}$ is completely contained in $\mathcal{A}_q^{ss}$ under these assumptions. Therefore any $4$-tuple $(A,\lambda, \iota, \xi)$ in $\mathcal{M}(q-2,2)_{\gamma_{u,v}}$ has the property that $(A, \lambda, \iota)$ is supersingular. This implies that the $4$-tuple is supersingular. 
\end{proof}

\subsection[Ekedahl-Oort strata on which F and V are not nilpotent]{Ekedahl-Oort strata on which $\mathbold{F}$ and $\mathbold{V}$ are not nilpotent}

If $A[p]$ has a non-trivial subgroup scheme on which $F$ or $V$ act bijectively, then the slopes $0$ and $1$ occur in the Newton polygon of $A$. Thus $F$ and $V$ act nilpotently on supersingular abelian varieties (equivalently, their so-called \emph{$p$-rank} is zero). By duality, $F$ is nilpotent if and only if $V$ is nilpotent. The following lemma shows how this is measured by the permutation $\omega_{u,v}$.

\begin{lemma} \label{lem:SMVFVnilp}
The action of $F$ is nilpotent on $M$ if and only if $\omega_{u,v}(1)=1$.
\end{lemma}
\begin{proof}
Assume the action of $F$ is nilpotent on $M$. Then we must have $F(W_1)=0$, because otherwise $F$ acts bijectively on $W_1$. Hence
$$\eta(1)=\dim (W_1 \cap M[F]) = 1$$ and, therefore, $j_1=1$ and $\omega_{u,v}(1)=1$.

On the other hand, assume that $F$ is not nilpotent on $M$. This implies that $F^n(M) \neq 0$ for every $n$. Since applying $F$ gives a subspace in the canonical filtration, we infer that there is an $l>0$ such that $F(W_l)=W_l$. Thus $\eta(l)=0$ and, therefore, $i_1=1$ and $\omega_{u,v}(1)=q+1$. 
\end{proof}

This allows us show that several Ekedahl-Oort strata cannot intersect the supersingular locus.

\begin{proposition} \label{prop:ruledoutbyFVnilp}
Assume $u>1$. Then the Ekedahl-Oort stratum $\mathcal{M}(q-2,2)_{\gamma_{u,q}}$ does not intersect $\mathcal{M}(q-2,2)^{ss}$.
\end{proposition}
\begin{proof}
By \cref{lem:SMVFVnilp}, it suffices to check that $\omega_{u,q}(1) >1$. Appealing to \cref{thm:SiegelWeyl}, we conclude that this happens exactly when $u>1$ and $v=q$. Since $(\mathcal{A}_{q})_{\omega_{u,q}}$ does not intersect $\mathcal{A}_q^{ss}$, it follows that $\mathcal{M}(q-2,2)_{\gamma_{u,q}}$ does not intersect $\mathcal{M}(q-2,2)^{ss}$. 
\end{proof}

\subsection{Minimal Ekedahl-Oort strata of non-supersingular Newton strata}

There is another method that shows when Ekedahl-Oort strata are disjoint from a given Newton stratum. We focus on interactions with the supersingular locus, though the techniques can be applied more generally. This method is based on \emph{minimal} Ekedahl-Oort strata, which are completely contained in a non-supersingular Newton stratum.

\begin{definition} \label{def:EOminimal}
An Ekedahl-Oort stratum $S$ of a Shimura variety of PEL type $\mathcal{M}$ is \emph{minimal} if $A[p^\infty] \cong B[p^\infty]$ holds for any $A,B \in S(\kk)$.
\end{definition}

It follows immediately from this definition that the Ekedahl-Oort stratum $S$ is completely contained in one Newton stratum. Unfortunately, it is not known what the minimal Ekedahl-Oort strata of $\mathcal{M}(q-2,2)$ look like, or whether each Newton stratum of $\mathcal{M}(q-2,2)$ contains a (unique) minimal Ekedahl-Oort stratum, since the group $\mathsf{GU}(q-2,2)$ is not split for $q > 4$ (see \cite{ViehmannWedhorn2013}).

Luckily things are different for $\mathcal{A}_q$. Oort proved in \cite{Oort05minimal,Oort05simple} that each Newton stratum of $\mathcal{A}_q$ contains a unique minimal Ekedahl-Oort stratum. Given a Newton polygon, results of \cite[5.3]{DeJongOort} give an explicit description of the $p$-divisible group and the $p$-torsion group scheme of the minimal Ekedahl-Oort stratum contained in that Newton stratum. Hence, the minimal Ekedahl-Oort stratum of a non-supersingular Newton stratum does not intersect $\mathcal{A}_q^{ss}$, and we can use \cref{thm:SiegelWeyl} to determine Ekedahl-Oort strata of $\mathcal{M}(q-2,2)$ that do not intersect $\mathcal{M}(q-2,2)^{ss}$.

\begin{definition} \label{def:SMVMmn}
Let $m$ and $n$ be non-negative integers. Define the mod-$p$ Dieudonn\'{e} module $M_{m,n} := \Span_\kk \{e_0, \ldots , e_{m+n-1} \}$ with following action of $F$ and $V$:
\begin{align}
    F(e_i) &= \begin{cases} e_{i+n} & \hbox{if $i \leq m-1$}, \\ 0 &\hbox{else}, \end{cases} \label{eq:MmnF} \\
    V(e_i) &= \begin{cases} e_{i+m} & \hbox{if $i \leq n-1$}, \\ 0 &\hbox{else.} \end{cases} \label{eq:MmnV}
\end{align}
\end{definition}

Let $\alpha$ be a Newton stratum of $\mathcal{A}_q$ with slopes $\frac{n_1}{m_1+n_1} , \ldots, \frac{n_r}{m_r+n_r}$, where each slope may occur with multiplicity greater than or equal to 1. Oort shows in \cite{Oort05minimal} that the unique minimal Ekedahl-Oort stratum in $\alpha$ has the Dieudonn\'{e} module
\begin{equation} \label{eq:Malpha}
M_\alpha := \bigoplus_{l=1}^r M_{m_l, n_l}.
\end{equation}
The Dieudonn\'{e} module $M_\alpha$ corresponds to a permutation $\omega_\alpha \in \mathbf{W}_q$, which we determine in the results below.

\begin{lemma} \label{lem:Wmn}
$M_{m,n}$ has a final filtration given by 
$$W_j^{m,n}:= \Span_\kk \{e_{m+n-j} , \ldots , e_{m+n-1} \},$$
where
\begin{align} 
    F(W_j^{m,n}) &= W^{m,n}_{\max\{0,j-n\}}, \label{eq:WmnF} \\
    V^{-1}(W_j^{m,n}) &= W^{m,n}_{\min\{m+n,j+m\}}, \label{eq:WmnV} \\
    M_{m,n}[F] &= W^{m,n}_{n}. \label{eq:Wmn[F]}
\end{align} 
\end{lemma}
\begin{proof}
Equation~\eqref{eq:WmnF} follows from Equation~\eqref{eq:MmnF}; all  indices are shifted by $n$ when $F$ is applied, except if the resulting index exceeds $m+n-1$, in which case $F$ acts as $0$.

Similarly, Equation~\eqref{eq:WmnV} is proved using Equation~\eqref{eq:MmnV}. Since the filtration $W^{m,n}_\bullet$ is stable under $F$ and $V^{-1}$ and $\dim_\kk(W_j)=j$, it is a final filtration of $M_{m,n}$. Finally, Equation \eqref{eq:Wmn[F]} also follows directly from Equation \eqref{eq:MmnF}.
\end{proof}
Given integers $m,n$ and a word $w$ in the letters $F$ and $V^{-1}$, define the integer $w(m,n)$ recursively as follows:
\begin{itemize}
    \item If $w$ is the empty word, then $w(m,n)=m+n$;
    \item $(Fw)(m,n) = \max\{0,w(m,n)-n\}$;
    \item $(V^{-1}w)(m,n)= \min\{m+n, w(m,n)+m\}$.
\end{itemize}

\begin{corollary} \label{cor:w(Mmn)}
We have $w(M_{m,n}) = W^{m,n}_{w(m,n)}$.
\end{corollary}
\begin{proof}
Apply \cref{lem:Wmn} repeatedly. 
\end{proof}

Recall the definition of the mod-$p$ Dieudonn\'{e} module $M_\alpha$ from Equation \eqref{eq:Malpha}. Combining this with \cref{cor:w(Mmn)} gives, for any word $w$ in the alphabet $\{F,V^{-1}\}$, the formula
\[
    w(M_\alpha) = w \left( \bigoplus_{l=1}^r M_{m_l,n_l}  \right) 
    = \bigoplus_{l=1}^r w(M_{m_l,n_l}) 
    = \bigoplus_{l=1}^r W^{m_l,n_l}_{w(m_l,n_l)}.
\] 
 Now, let $W^\alpha_\bullet$ be a final filtration of $M_\alpha$ and define the function
$\eta_\alpha(j) = \dim \left( W^\alpha_j \cap M_\alpha[F] \right).$
We have the following restrictions on $\eta_\alpha$:
\begin{align*}
    \eta_\alpha\left(\sum_{l=1}^r w(m_l,n_l)\right) &= \dim_\kk \left( \left(\bigoplus_{l=1}^r W^{m_l,n_l}_{w(m_l,n_l)} \right) \cap M_\alpha [F] \right) \\
    &= \dim_\kk \left( \bigoplus_{l=1}^r W^{m_l,n_l}_{w(m_l,n_l)} \cap W^{m_l,n_l}_{n_l} \right) \\
    &= \sum_{l=1}^r \min \{n_l, w(m_l,n_l) \}.
\end{align*} 
By letting $w$ range over sufficiently many words, we obtain enough restrictions to determine the non-decreasing function $\eta_\alpha : \{1, \ldots , 2q \} \to \{ 1, \ldots , q\}$ uniquely. This is because $\eta_\alpha$ is determined by the canonical filtration of $M_\alpha$. Finally, this function $\eta_\alpha$ gives rise to a permutation $\omega_\alpha \in \mathbf{W}_q$ via the steps given in \cite[3.6]{Moonen2001}.

\begin{proposition} \label{prop:SMVruledoutbyminimal}
Let $n_1, \ldots ,n_r $ be non-negative integers satisfying the following conditions:
\begin{enumerate} [label= (\roman*)]
    \item $\sum_{l=1}^r n_l = q$;\label{item:sum=q}
    \item $\gcd(n_l,n_{r+1-l})=1$ for every $l \in \{1, \ldots , r\}$; \label{item:coprime}
    \item there exists $l \in \{1 ,\ldots ,r\}$ such that $n_l + n_{r+1-l} \neq 2$. \label{item:notss}
\end{enumerate}
Define the mod-$p$ Dieudonn\'{e} module $M_\alpha:= \bigoplus M_{n_{r+1-l},n_l}$ and let $\omega_\alpha \in \mathfrak{S}_{2q}$ be as above. Then the Ekedahl-Oort stratum $\mathcal{A}_{\omega_\alpha}$ does not intersect $\mathcal{A}_q^{ss}$.
\end{proposition}
\begin{proof}
By conditions \ref{item:sum=q} and \ref{item:coprime}, there exists a symmetric Newton polygon from $(0,0)$ to $(2q,q)$ with slopes $\frac{n_l}{n_l + n_{r+1-l}}$ (possibly with multiplicity greater than 1). By condition \ref{item:notss}, this Newton polygon has a slope that is not $1/2$. By construction, $M_\alpha$ and $\omega_\alpha$ correspond to a minimal Ekedahl-Oort stratum contained in a Newton stratum that is not supersingular. Thus it does not intersect $\mathcal{A}_q^{ss}$. 
\end{proof}

\begin{corollary} \label{cor:ruledoutbyminimal}
Suppose $\omega_{u,v} = \omega_\alpha$ under the conditions of \cref{prop:SMVruledoutbyminimal}. Then the Ekedahl-Oort stratum of $\mathcal{M}(q-2,2)_{\gamma_{u,v}}$ does not intersect $\mathcal{M}(q-2,2)^{ss}.$
\end{corollary}
\begin{proof}
This follows directly from \cref{prop:SMVruledoutbyminimal}, as (forgetting) the unitary structure does not affect whether an abelian variety is supersingular or not. 
\end{proof} 


\bigskip

\bibliography{citations.bib}{}
\bibliographystyle{amsalpha}

\end{document}